\documentclass[a4paper]{article}

\usepackage[inner=30mm,outer=30mm,textheight=225mm]{geometry}

\usepackage{amsmath}        % For mathematical symbols
\usepackage{amssymb}        % For mathematical symbols
\usepackage{enumerate}      % To support \begin{enumerate}[(i)]
\usepackage{graphicx}       % For embedding figures
\usepackage{theorem}        % For custom theorem styles

\newtheorem{theorem}{Theorem}[section]
\newtheorem{definition}[theorem]{Definition}
\newtheorem{lemma}[theorem]{Lemma}
\newtheorem{corollary}[theorem]{Corollary}
\newtheorem{algorithm}[theorem]{Algorithm}

\theorembodyfont{\upshape}
\newtheorem{example}[theorem]{Example}

\theorembodyfont{\upshape}
\newcommand{\prooflabel}{Proof}
\newcommand{\qed}{\nolinebreak\hfill\rule[-0.5mm]{1.5mm}{3.0mm}}
\newtheorem{proofthm}{\prooflabel}
\newenvironment{proof}{\begin{proofthm}}{\qed \end{proofthm}}

\newcommand{\chiqo}{\chi'}
\newcommand{\cobdry}{\delta}
\newcommand{\dimfix}{di\-men\-sion\-al}
\newcommand{\dualtri}{\mathcal{D}}
\newcommand{\efffix}{0-effi\-cient}
\newcommand{\extmap}{\varepsilon}
\newcommand{\jrep}[1]{\mathbf{j}(#1)}
\newcommand{\krep}[1]{\mathbf{k}(#1)}
\newcommand{\projmap}{\pi}
\newcommand{\projqo}{\mathcal{P}_\mathit{QO}}
\newcommand{\qrep}[1]{\mathbf{q}(#1)}
\newcommand{\quadfix}{quad\-ri\-lat\-eral}
\newcommand{\Quadfix}{Quad\-ri\-lat\-eral}
\newcommand{\QuadOct}{{\Quadfix}-Octa\-gon}
\newcommand{\Quadoct}{{\Quadfix}-octa\-gon}
\newcommand{\quadoct}{{\quadfix}-octa\-gon}
\newcommand{\R}{\mathbb{R}}
\newcommand{\regina}{\emph{Regina}}

\newcommand{\tri}{\mathcal{T}}
\newcommand{\trisphere}{\mathcal{S}}

\newcommand{\vrep}[1]{\mathbf{v}(#1)}
\newcommand{\Z}{\mathbb{Z}}

\newlength{\displaytextwidth}
\newlength{\displaytextindent}
\newcommand{\displaytext}[2]
    {\par\medskip\par
     \setlength{\displaytextwidth}{\linewidth}
     \addtolength{\displaytextwidth}{-\displaytextindent}
     \addtolength{\displaytextwidth}{-\displaytextindent}
     \noindent\hspace{\displaytextindent}\parbox[b]{\displaytextwidth}{%
        \textbf{#1: } \textit{#2}}%
     \par\medskip\par}

\title{Quadrilateral-Octagon Coordinates for \\ Almost Normal Surfaces}
\author{Benjamin A.~Burton}
\date{August 24, 2009}

\begin{document}

\maketitle

\abstract{Normal and almost normal surfaces are essential tools
for algorithmic 3-manifold topology, but to use them requires
exponentially slow enumeration algorithms in a high-dimensional vector space.
The quadrilateral coordinates of Tollefson alleviate this problem
considerably for normal surfaces, by reducing the dimension of this
vector space from $7n$ to $3n$ (where $n$ is the complexity of the
underlying triangulation).  Here we develop an analogous
theory for octagonal almost normal surfaces, using quadrilateral and octagon
coordinates to reduce this dimension from $10n$ to $6n$.
As an application, we show that {\quadoct} coordinates can be used exclusively
in the streamlined 3-sphere recognition algorithm of Jaco, Rubinstein
and Thompson, reducing experimental running times by factors of thousands.
We also introduce joint coordinates, a system with only $3n$ dimensions for
octagonal almost normal surfaces that has appealing geometric properties.}



\section{Introduction} \label{s-intro}

The theory of normal surfaces, introduced by Kneser \cite{kneser29-normal}
and developed by Haken \cite{haken61-knot,haken62-homeomorphism}, is central
to algorithmic 3-manifold topology.  In essence, normal surface theory
allows us to search for ``interesting'' embedded surfaces within a
3-manifold triangulation $\tri$ by enumerating the vertices of a
polytope in a high-dimensional vector space.  Normal surfaces are
defined by their intersections with the tetrahedra of $\tri$, which must
be collections of disjoint triangles and/or quadrilaterals, collectively
referred to as \emph{normal discs}.

In the early 1990s, Rubinstein introduced the concept of an almost normal
surface, for use with problems such as 3-sphere recognition and finding
Heegaard splittings \cite{rubinstein97-3sphere}.  Almost normal surfaces are
essentially normal surfaces with a single unusual intersection piece,
which may be either an octagon or a tube.  Thompson subsequently refined
the 3-sphere recognition algorithm to remove any need for tubes
\cite{thompson94-thinposition}, and since then almost normal surfaces
have appeared
in algorithms such as determining Heegaard genus \cite{lackenby08-tunnel},
recognising small Seifert fibred spaces \cite{rubinstein04-smallsfs}, and
finding bridge surfaces in knot complements \cite{wilson08-anknot}.

In this paper we focus on \emph{octagonal} almost normal surfaces; that
is, almost normal surfaces in which the unusual intersection piece is an
octagon, not a tube.  The reason for this restriction is that octagonal
almost normal surfaces are both tractable and useful, and have important
applications beyond 3-manifold topology.  In detail:
\begin{itemize}
    \item For practical computation, octagonal almost normal surfaces
    are significantly easier to deal with than general almost normal
    surfaces.  In particular, the translation between surfaces and
    high-dimensional vectors becomes much simpler, and the enumeration
    of these vectors is less fraught with complications.
    \item As shown by Thompson \cite{thompson94-thinposition},
    octagonal almost normal surfaces are sufficient for running the
    3-sphere recognition algorithm.
    \item Following on from the previous point,
    an efficient 3-sphere recognition algorithm is important for
    computation in \emph{4-manifold} topology.  For example, answering
    even the basic question ``is $\tri$ a 4-manifold
    triangulation?''\ requires us to run the 3-sphere recognition
    algorithm over a neighbourhood of each vertex of $\tri$.
    Therefore, improving the efficiency of 3-sphere recognition
    is an important step towards a general efficient computational
    framework for working with 4-manifold triangulations.
\end{itemize}

As suggested above,
our focus here is on the \emph{efficiency} of working with almost normal
surfaces.  The fundamental problem that we face is that the underlying polytope
vertex enumeration can grow exponentially slowly in the number of
tetrahedra.  This means that in practice normal surface algorithms
cannot be run on large triangulations.
Moreover, this exponential growth is not the fault of the algorithms,
but an unavoidable feature of the problems that they try to solve.  For
illustrations of this, see \cite{burton09-extreme} which describes cases
in which the underlying vertex enumeration problem has exponentially
many solutions, or see the proof by Agol et al.~that
computing 3-manifold knot genus (one of the many applications of normal
surface theory) is NP-complete \cite{agol02-knotgenus}.

For almost normal surfaces, our efficiency troubles are even worse than
for normal surfaces.  This is because the polytope vertex enumeration is
not just exponentially slow in the number of tetrahedra $n$, but also in
the dimension of the underlying vector space.  For normal surfaces this
dimension is $7n$, whereas
for octagonal almost normal surfaces this dimension is
$10n$, a significant difference when dealing with an exponential algorithm.

In the realm of normal surfaces, much progress has been made in improving
the efficiency of enumeration algorithms
\cite{burton08-dd,burton09-convert,tollefson98-quadspace}.
One key development has been Tollefson's quadrilateral coordinates
\cite{tollefson98-quadspace}, in which we work only with
quadrilateral normal discs and then reconstruct the triangular
discs afterwards.  This allows us to perform our expensive polytope
vertex enumeration in dimension $3n$ instead of $7n$, which yields
substantial efficiency improvements.

There are two complications with Tollefson's approach:
\begin{itemize}
    \item When reconstructing a normal surface from its
    quadrilateral discs, we cannot recover any vertex linking components
    (these components lie at the frontiers of small
    regular neighbourhoods of vertices of the triangulation).  This is
    typically not a problem, since such components are rarely of interest.

    \item When we use quadrilateral coordinates for the
    underlying polytope vertex enumeration, some solutions are ``lost''.
    That is, the resulting set of normal surfaces
    (called \emph{vertex normal surfaces}) is a strict subset of what we
    would obtain using the traditional $7n$-dimensional framework of Haken.

    This latter issue can be resolved in two different ways.
    For some high-level topological algorithms,
    such as the detection of two-sided incompressible surfaces
    \cite{tollefson98-quadspace}, it has been proven that at least one of
    the surfaces that we need to find will not be lost.
    As a more general resolution to this problem, there is a fast
    quadrilateral-to-standard conversion algorithm through which we
    can recover all of the lost surfaces \cite{burton09-convert}.
\end{itemize}

The main purpose of this paper is to develop an analogous theory for octagonal
almost normal surfaces.  Specifically, we show that we can work with only
quadrilateral normal discs and octagonal almost normal discs, and then
reconstruct the triangular discs afterwards.  As a consequence, the
dimension for our vertex enumeration drops from $10n$ to $6n$.

We run into the same complications as before---vertex linking components
cannot be recovered, and we may lose some of our original solutions.
Here we show that, as with quadrilateral coordinates, these are not serious
problems.  In particular, we show that despite this loss of information,
{\quadoct} coordinates suffice for the 3-sphere recognition
algorithm.  More generally, we observe that
the fast quadrilateral-to-standard conversion algorithm of
\cite{burton09-convert} works seamlessly with octagonal almost normal surfaces.

As a practical measure of benefit, we use the software package {\regina}
\cite{regina,burton04-regina} to compare running times for the 3-sphere
recognition algorithm with and without {\quadoct} coordinates.
Here we see {\quadoct} coordinates improving
performance by factors of thousands in several cases.
Readers can experiment with {\quadoct} coordinates for themselves by
downloading {\regina} version~4.6 or later.

We finish this paper by introducing
\emph{joint coordinates}, in which we exploit natural relationships
between quadrilaterals and octagons to reduce our $6n$ dimensions for
octagonal almost normal surfaces down to just $3n$ dimensions.  Although
these coordinates cannot be used with existing enumeration algorithms
(due to a loss of convexity in the underlying polytope), they have
appealing geometric properties that make them useful for visualisation,
and which may help develop intuition about the structure of the
corresponding solution space.

All of the results in this paper apply only to compact 3-manifold
triangulations.  In particular, they do not cover the ideal triangulations
of Thurston \cite{thurston78-lectures}, where the reconstruction of triangular
discs can result in pathological (but nevertheless useful) objects such as
spun normal surfaces, which contain infinitely many discs
\cite{tillmann08-finite}.

The layout of this paper is as follows.
Section~\ref{s-normal} begins with an overview of normal surfaces and
Tollefson's quadrilateral coordinates, and Section~\ref{s-almost}
follows with an overview of almost normal surfaces.
In Section~\ref{s-quadoct} we develop the core theory for {\quadoct}
coordinates, including necessary and sufficient conditions for a
$6n$-dimensional vector to represent an octagonal almost normal surface.

For the remainder of the paper we focus on applications and extensions of
this theory.
In Section~\ref{s-sphere} we describe the streamlined 3-sphere
recognition algorithm of Jaco, Rubinstein and Thompson
\cite{jaco03-0-efficiency}, and show that this algorithm
remains correct when we work in {\quadoct} coordinates
instead of the original $10n$-dimensional vector space.
Section~\ref{s-enumeration} focuses on the underlying polytope vertex
enumeration algorithm, where we observe that state-of-the-art algorithms
for enumerating normal surfaces \cite{burton08-dd,burton09-convert}
can be used seamlessly with octagonal almost normal surfaces and {\quadoct}
coordinates.
In Section~\ref{s-performance} we offer experimental measures of running
time that show how {\quadoct} coordinates improve the 3-sphere
recognition algorithm in practice, and in
Section~\ref{s-joint} we finish with a discussion of joint coordinates.

The author is grateful to the Victorian Partnership for Advanced
Computing for the use of their excellent computing resources, to the
University of Melbourne for their continued support for the software
package {\regina}, and to the anonymous referees for their thoughtful
suggestions.

\section{Normal Surfaces} \label{s-normal}

We assume that the reader is already familiar with the theory of normal
surfaces (if not, a good overview can be found in \cite{hass99-knotnp}).
In this section we outline the relevant aspects of the theory, concentrating
on the differences between Haken's original formulation \cite{haken61-knot}
and Tollefson's quadrilateral coordinates \cite{tollefson98-quadspace}.
For a more detailed discussion of these two formulations and the
relationships between them, the reader is referred to \cite{burton09-convert}.

Throughout this paper we assume that we are working with a
compact 3-manifold triangulation $\tri$ formed from $n$ tetrahedra.
By a \emph{compact triangulation},
we mean that every vertex of $\tri$ has a small neighbourhood whose frontier
is a sphere or a disc.  This ensures that $\tri$ is a triangulation of a
compact 3-manifold (possibly with boundary), and rules out the ideal
triangulations of Thurston \cite{thurston78-lectures} in which vertices
form higher-genus cusps.

To help keep the number of tetrahedra in $\tri$
small, we allow different faces of a tetrahedron to be identified (and
likewise with edges and vertices).  Some authors refer to triangulations
with this property as \emph{pseudo-triangulations} or
\emph{semi-simplicial triangulations}.
Faces, edges and vertices of $\tri$ that lie entirely within the
3-manifold boundary are called \emph{boundary faces},
\emph{boundary edges} and \emph{boundary vertices} of $\tri$ respectively.

An \emph{embedded normal surface} in $\tri$ is a properly embedded surface
(possibly disconnected or empty)
that intersects each tetrahedron of $\tri$ in a collection of disjoint
\emph{normal discs}.  Each normal disc is either a \emph{triangle}
or a \emph{quadrilateral}, with a boundary consisting of three or four
arcs respectively that cross distinct faces of the tetrahedron.
Figure~\ref{fig-normaldiscs} illustrates several disjoint triangles and
quadrilaterals within a tetrahedron.

\begin{figure}[htb]
\centering
\includegraphics[scale=0.7]{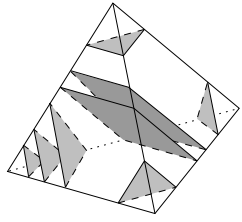}
\caption{Several normal discs within a single tetrahedron}
\label{fig-normaldiscs}
\end{figure}

The triangles and quadrilaterals within a tetrahedron can be grouped
into seven \emph{normal disc types}, according to which edges of the
tetrahedron they intersect.  This includes four triangular disc types
and three quadrilateral disc types, all of which are illustrated in
Figure~\ref{fig-normaltypes}.

\begin{figure}[htb]
\centering
\includegraphics[scale=0.55]{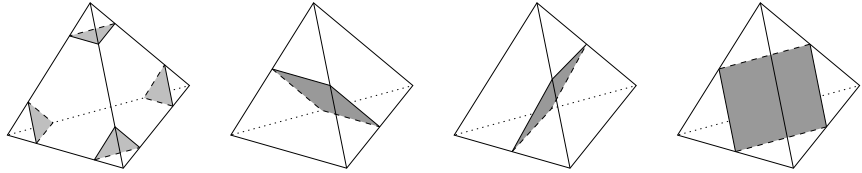}
\caption{The seven different normal disc types within a tetrahedron}
\label{fig-normaltypes}
\end{figure}

Equivalence of normal surfaces is defined by \emph{normal isotopy}, which
is an ambient isotopy that preserves each simplex of the
triangulation $\tri$.  Throughout this paper, any two surfaces that are
related by normal isotopy are regarded as the same surface.

Vertex links are normal surfaces that play an important role in the
discussion that follows.  If $V$ is a vertex of the triangulation $\tri$
then the \emph{vertex link} of $V$, denoted $\ell(V)$, is the normal surface
at the frontier of a small regular neighbourhood of $V$.  This surface
is formed entirely from triangular discs (one copy of each triangular
disc type surrounding $V$).  Here we follow the nomenclature of
Jaco and Rubinstein \cite{jaco03-0-efficiency}; Tollefson refers to
vertex links as \emph{trivial surfaces}.

A core strength of normal surface theory is its ability to reduce
difficult problems in topology to simpler problems in linear algebra.
This is where the formulations of Haken and Tollefson differ, and so we
slow down from here onwards to give full details.  The key difference
between the two formulations is that Haken works in a $7n$-dimensional
vector space with coordinates based on triangle and quadrilateral disc types,
whereas Tollefson works in a $3n$-dimensional space based on quadrilateral
disc types only.

\begin{definition}[Vector Representations]
    Let $\tri$ be a compact 3-manifold triangulation formed from the $n$
    tetrahedra $\Delta_1,\ldots,\Delta_n$, and let $S$ be an embedded
    normal surface in $\tri$.  For each tetrahedron $\Delta_i$, let
    $t_{i,1}$, $t_{i,2}$, $t_{i,3}$ and $t_{i,4}$ denote the number of
    triangular discs of $S$ of each type in $\Delta_i$, and let
    $q_{i,1}$, $q_{i,2}$ and $q_{i,3}$ denote the number of
    quadrilateral discs of $S$ of each type in $\Delta_i$.

    Then the \emph{standard vector representation} of $S$, denoted
    $\vrep{S}$, is the $7n$-dimensional vector
    \[ \vrep{S}~=\,
    \left(\ t_{1,1},t_{1,2},t_{1,3},t_{1,4},\ q_{1,1},q_{1,2},q_{1,3}\ ;
    \ t_{2,1},t_{2,2},t_{2,3},t_{2,4},\ q_{2,1},q_{2,2},q_{2,3}\ ;
    \ \ldots,q_{n,3}\ \right), \]
    and the \emph{quadrilateral vector representation} of $S$,
    denoted $\qrep{S}$, is the $3n$-dimensional vector
    \[ \qrep{S}~=\,
    \left(\ q_{1,1},q_{1,2},q_{1,3}\ ;\ q_{2,1},q_{2,2},q_{2,3}\ ;
    \ \ldots,q_{n,3}\ \right). \]
\end{definition}

When we are working with $\vrep{S}$, we say we are
working in \emph{standard coordinates} (or \emph{standard normal
coordinates} if we wish to distinguish between normal and almost normal
surfaces).  Likewise, when working with $\qrep{S}$ we say we are working
in \emph{quadrilateral coordinates}.  The following uniqueness results
are due to Haken \cite{haken61-knot} and Tollefson \cite{tollefson98-quadspace}:

\begin{lemma} \label{l-vecrep}
    Let $\tri$ be a compact 3-manifold triangulation, and let
    $S$ and $S'$ be embedded normal surfaces in $\tri$.
    \begin{itemize}
        \item The standard vector representations
        $\vrep{S}$ and $\vrep{S'}$ are equal if and only if
        the surfaces $S$ and $S'$ are normal isotopic
        (i.e., they are the ``same'' normal surface).

        \item The quadrilateral vector representations
        $\qrep{S}$ and $\qrep{S'}$ are equal if and only if
        either (i)~$S$ and $S'$ are normal isotopic, or
        (ii)~$S$ and $S'$ can be made normal isotopic by
        adding or removing vertex linking components.
    \end{itemize}
\end{lemma}

Since we are rarely interested in vertex linking components,
Lemma~\ref{l-vecrep} shows that the standard and quadrilateral
vector representations each contain everything we might want to know
about an embedded normal surface.

Not every integer vector $\mathbf{w} \in \R^{7n}$ or
$\mathbf{w} \in \R^{3n}$ is the vector representation of a normal surface.
The necessary conditions on $\mathbf{w}$ include a set of
\emph{matching equations} as well as a set of \emph{quadrilateral constraints},
which we define as follows.

\begin{definition}[Standard Matching Equations] \label{d-matchingstd}
    Let $\tri$ be a compact 3-manifold triangulation formed from the $n$
    tetrahedra $\Delta_1,\ldots,\Delta_n$, and let $\mathbf{w} \in \R^{7n}$
    be any $7n$-dimensional vector whose coordinates we label
    \[ \mathbf{w}~=\,
    \left(\ t_{1,1},t_{1,2},t_{1,3},t_{1,4},\ q_{1,1},q_{1,2},q_{1,3}\ ;
    \ t_{2,1},t_{2,2},t_{2,3},t_{2,4},\ q_{2,1},q_{2,2},q_{2,3}\ ;
    \ \ldots,q_{n,3}\ \right). \]
    For each non-boundary face of $\tri$ and each of the three edges
    surrounding it, we obtain a \emph{standard matching equation}
    on $\mathbf{w}$ as follows.

    Let $F$ be some non-boundary face of $\tri$, and let $e$ be one
    of the three edges surrounding $F$.  Suppose that
    $\Delta_i$ and $\Delta_j$ are the two tetrahedra on either side of $F$.
    Then there is precisely one triangular disc type and one quadrilateral
    disc type in each of $\Delta_i$ and $\Delta_j$ that meets $F$ in an
    arc parallel to $e$, as illustrated in Figure~\ref{fig-matchingstd}.
    Suppose these disc types correspond to coordinates
    $t_{i,a}$, $q_{i,b}$, $t_{j,c}$ and $q_{j,d}$ respectively.
    Then we obtain the matching equation
    \[ t_{i,a} + q_{i,b} = t_{j,c} + q_{j,d}. \]
\end{definition}

\begin{figure}[htb]
\centering
\includegraphics{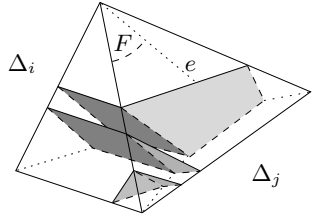}
\caption{Building the standard matching equations}
\label{fig-matchingstd}
\end{figure}

Essentially, the standard matching equations ensure that all of the normal
discs on either side of a non-boundary face $F$ can be joined together.
In Figure~\ref{fig-matchingstd}, the four coordinates are
$(t_{i,a},q_{i,b},t_{j,c},q_{j,d}) = (1,2,2,1)$, giving the
equation $1+2=2+1$ which is indeed satisfied.
If $\tri$ is a closed triangulation (i.e., it has no boundary), then
there are precisely $6n$ standard matching equations for $\tri$
(three for each of the $2n$ faces of $\tri$).

\begin{definition}[Quadrilateral Matching Equations] \label{d-matchingquad}
    Let $\tri$ be a compact 3-manifold triangulation formed from the $n$
    tetrahedra $\Delta_1,\ldots,\Delta_n$, and let $\mathbf{w} \in \R^{3n}$
    be any $3n$-dimensional vector whose coordinates we label
    \[ \mathbf{w}~=\,
    \left(\ q_{1,1},q_{1,2},q_{1,3}\ ;\ q_{2,1},q_{2,2},q_{2,3}\ ;
    \ \ldots,q_{n,3}\ \right). \]
    For each non-boundary edge of $\tri$, we obtain a
    \emph{quadrilateral matching equation}
    on $\mathbf{w}$ as follows.

    Let $e$ be some non-boundary edge of $\tri$, and arbitrarily label
    the two ends of $e$ as \emph{upper} and \emph{lower}.
    The tetrahedra containing edge $e$ are arranged in a cycle around $e$,
    as illustrated in Figure~\ref{fig-matchingquad}.  Choose some arbitrary
    direction around this cycle, and suppose that the tetrahedra that we
    encounter as we travel in this direction around the cycle are labelled
    $\Delta_{i_1},\ldots,\Delta_{i_t}$.

    \begin{figure}[htb]
    \centering
    \includegraphics[scale=0.8]{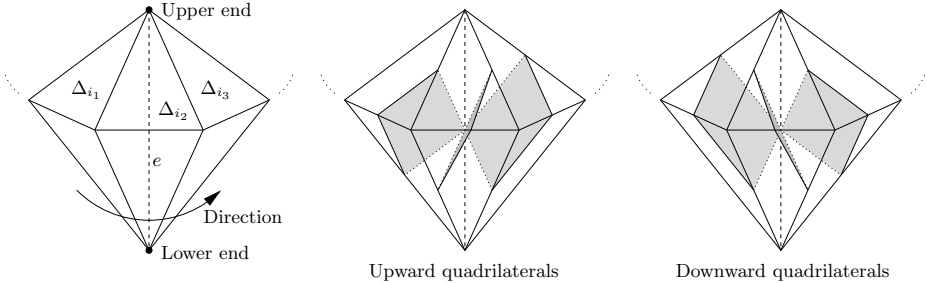}
    \caption{Building the quadrilateral matching equations}
    \label{fig-matchingquad}
    \end{figure}

    For each tetrahedron in this cycle, there are two quadrilateral
    types meeting edge $e$: one that rises from the lower end of $e$
    to the upper as we travel around the cycle in the chosen direction,
    and one that falls from the upper end of $e$ to the lower.
    We call these the \emph{upward quadrilaterals} and \emph{downward
    quadrilaterals} respectively; these are again illustrated in
    Figure~\ref{fig-matchingquad}.

    Suppose now that the coordinates corresponding to the upward and
    downward quadrilateral types are
    $q_{i_1,u_1},q_{i_2,u_2},\ldots,q_{i_t,u_t}$ and
    $q_{i_1,d_1},q_{i_2,d_2},\ldots,q_{i_t,d_t}$ respectively.
    Then we obtain the matching equation
    \begin{equation} \label{eqn-matchingquad}
       q_{i_1,u_1} + q_{i_2,u_2} + \ldots + q_{i_t,u_t} =
       q_{i_1,d_1} + q_{i_2,d_2} + \ldots + q_{i_t,d_t} .
    \end{equation}
    In other words, the total number of upward quadrilaterals surrounding $e$
    equals the total number of downward quadrilaterals surrounding $e$.
\end{definition}

Note that a single tetrahedron might appear multiple times in the cycle
around $e$, in which case a single coordinate $q_{i,j}$ might appear
more than once in the equation (\ref{eqn-matchingquad}).
For a closed triangulation $\tri$ with $v$ vertices, a quick
Euler characteristic calculation shows that we have precisely
$n+v$ edges in our triangulation and therefore precisely
$n+v$ quadrilateral matching equations.

\begin{definition}[Quadrilateral Constraints] \label{d-quadconst}
    \ Let $\tri$ be a compact 3-manifold triangulation formed from the $n$
    tetrahedra $\Delta_1,\ldots,\Delta_n$, and consider any vector
    \begin{align*}
    \mathbf{w}~&=\,
    \left(\ t_{1,1},t_{1,2},t_{1,3},t_{1,4},\ q_{1,1},q_{1,2},q_{1,3}\ ;
    \ \ldots,q_{n,3}\ \right) \in \R^{7n}\ \mbox{or} \\
    \mathbf{w}~&=\,
    \left(\ q_{1,1},q_{1,2},q_{1,3}\ ;
    \ \ldots,q_{n,3}\ \right) \in \R^{3n}.
    \end{align*}
    We say that $\mathbf{w}$ satisfies the \emph{quadrilateral
    constraints} if, for every tetrahedron $\Delta_i$,
    at most one of the quadrilateral coordinates
    $q_{i,1}$, $q_{i,2}$ and $q_{i,3}$ is non-zero.
\end{definition}

We can now describe a full set of necessary and sufficient conditions
for a vector $\mathbf{w} \in \R^{7n}$ or $\mathbf{w} \in \R^{3n}$ to be
the vector representation of some embedded normal surface.  The
following result is due to Haken \cite{haken61-knot} and
Tollefson \cite{tollefson98-quadspace}.

\begin{theorem} \label{t-admissible}
    Let $\tri$ be a compact 3-manifold triangulation formed from $n$ tetrahedra.
    An integer vector ($\mathbf{w} \in \R^{7n}$ or $\mathbf{w} \in \R^{3n}$)
    is the (standard or quadrilateral) vector representation of an
    embedded normal surface in $\tri$ if and only if:
    \begin{itemize}
        \item The coordinates of $\mathbf{w}$ are all non-negative;
        \item $\mathbf{w}$ satisfies the (standard or quadrilateral)
        matching equations for $\tri$;
        \item $\mathbf{w}$ satisfies the quadrilateral constraints for $\tri$.
    \end{itemize}
    Such a vector is referred to as an \emph{admissible
    vector}.\footnote{It is sometimes useful to extend the concept of
    admissibility to rational vectors or even real vectors in
    $\R^{7n}$ or $\R^{3n}$, as seen for instance in
    \cite{burton09-convert}.  However, we do not need such extensions in
    this paper.}
\end{theorem}

Essentially, the non-negativity constraint ensures that the coordinates of
$\mathbf{w}$ can be used to count normal discs, the matching equations
ensure that these discs can be joined together to form a surface, and
the quadrilateral constraint ensures that this surface is embedded
(since any two quadrilaterals of different types within the same
tetrahedron must intersect).

Many high-level algorithms in 3-manifold topology involve the
enumeration of \emph{vertex normal surfaces}, which form a basis from
which we can reconstruct all embedded normal surfaces within a
triangulation $\tri$.  The relevant definitions are as follows.

\begin{definition}[Projective Solution Space]
    \ Let $\tri$ be a compact 3-manifold triangulation formed from $n$ tetrahedra.
    The set of all non-negative vectors in $\R^{7n}$ that satisfy the
    standard matching equations for $\tri$ forms a rational polyhedral cone
    in $\R^{7n}$.  The \emph{standard projective solution space} for $\tri$
    is the rational polytope formed by intersecting this cone with the
    hyperplane $\{\mathbf{w} \in \R^{7n}\,|\,\sum w_i = 1\}$.

    The \emph{quadrilateral projective solution space} for $\tri$ is
    defined in a similar fashion by working in $\R^{3n}$ and using the
    quadrilateral matching equations instead.
\end{definition}

\begin{definition}[Vertex Normal Surface] \label{d-vertex}
    Let $\tri$ be a compact 3-manifold triangulation,
    and let $S$ be an embedded normal surface in $\tri$.
    If the standard vector representation $\vrep{S}$ is a positive
    multiple of some vertex of the standard projective solution space,
    then we call $S$ a \emph{standard vertex normal surface}.
    Likewise, if the quadrilateral vector representation $\qrep{S}$ is a
    positive multiple of some vertex of the quadrilateral projective solution
    space, then we call $S$ a \emph{quadrilateral vertex normal surface}.
\end{definition}

It should be noted that the definition of a vertex normal surface varies
between authors.  Definition~\ref{d-vertex} is consistent with
Jaco and Rubinstein \cite{jaco03-0-efficiency}, as well as
earlier work of this author \cite{burton09-convert}.
Other authors impose additional conditions, such as Tollefson
\cite{tollefson98-quadspace} who requires $S$ to be connected and
two-sided, or Jaco and Oertel \cite{jaco84-haken} who require the
elements of $\vrep{S}$ to have no common factor
(and who use the alternate name \emph{fundamental edge surface}).

Although vertex normal surfaces can be used as a basis for
reconstructing all embedded normal surfaces within a triangulation,
this is typically not feasible since there are infinitely many such
surfaces.  Instead we frequently find that, when searching for an
embedded normal surface with some desirable property, we can
restrict our attention \emph{only} to vertex normal surfaces.
For instance, Jaco and Oertel \cite{jaco84-haken} prove for closed
irreducible 3-manifolds that if a
two-sided incompressible surface exists then one can be found as a
standard vertex normal surface.
Likewise, Jaco and Tollefson \cite{jaco95-algorithms-decomposition}
prove that if a 3-manifold contains an essential disc or sphere then
one can be found as a standard vertex normal surface.

Using results of this type, a typical high-level algorithm based on normal
surface theory includes the following steps:
\begin{enumerate}[(i)]
    \item \label{en-algm-enumerate}
    Enumerate the (finitely many) vertices of the projective
    solution space for a given triangulation $\tri$, using techniques
    from linear programming (see \cite{burton08-dd} for details).

    \item Eliminate those vertices that do not satisfy the quadrilateral
    constraints, and then reconstruct the vertex normal surfaces of
    $\tri$ by taking multiples of those vertices that remain.
    Although there are infinitely many such multiples, only finitely many
    will yield \emph{connected} normal surfaces, which is typically what we
    are searching for.

    \item Test each of these vertex normal surfaces for some desirable
    property (such as incompressibility, or being an essential disc or
    sphere).
\end{enumerate}

Here we can see the real benefit of working in quadrilateral
coordinates---the enumeration of step~(\ref{en-algm-enumerate}) takes
place in a vector space of dimension $7n$ for standard coordinates, but
only $3n$ for quadrilateral coordinates.  Since both the running time and
memory usage can become exponential in this dimension \cite{burton08-dd},
a reduction from $7n$ to $3n$ can yield dramatic improvements in
performance.

However, there is a trade-off for using quadrilateral coordinates.
Although every connected quadrilateral vertex normal surface is also a
standard vertex normal surface \cite{burton09-convert}, the converse is not
true in general.  Instead, there might be standard vertex normal surfaces
(perhaps including the incompressible surfaces, essential discs and spheres
or whatever else we are searching for) that do not show up as
quadrilateral vertex normal surfaces.  These ``lost surfaces'' can undermine
the correctness of our algorithms, which we maintain in one of two ways:
\begin{itemize}
    \item
    We can resolve the problem using theory.  This requires
    us to prove that, if the surface we are searching for exists, then
    it exists not only as a \emph{standard} vertex normal surface but
    also as a \emph{quadrilateral} vertex normal surface.

    Such results can be more difficult to prove in quadrilateral coordinates
    than in standard coordinates,
    partly because important functions such as Euler characteristic are
    no longer linear.  Nevertheless, examples can be
    found---Tollefson \cite{tollefson98-quadspace} proves such a result
    for two-sided incompressible surfaces, and
    Jaco et al.~\cite{jaco02-algorithms-essential} refer to similar
    results for essential discs and spheres.

    \item
    We can resolve the problem using algorithms and computation.
    There is a fast algorithm described in \cite{burton09-convert}
    that converts a full set of quadrilateral vertex normal surfaces to a
    full set of standard vertex normal surfaces, thereby recovering those
    surfaces that were lost.
    This algorithm is found to have a negligible running time, which means
    that we are able to work with standard
    vertex normal surfaces yet still enjoy the significantly greater
    performance of quadrilateral coordinates.
\end{itemize}

The main part of this paper is concerned with the development of
{\quadoct} coordinates for almost normal surfaces, where we face a
similar trade-off.  In Section~\ref{s-sphere} we resolve this problem
for the 3-sphere recognition algorithm using the theoretical route, and in
Section~\ref{s-enumeration} we show how the more general algorithmic solution
can be used.

\section{Almost Normal Surfaces} \label{s-almost}

Almost normal surfaces are an extension of normal surfaces
whereby, in addition to the usual normal discs, we allow one tetrahedron
of the triangulation to contain a single unusual intersection piece.
Introduced by Rubinstein for use with the 3-sphere algorithm and related
problems \cite{rubinstein95-3sphere,rubinstein97-3sphere},
almost normal surfaces also enjoy other applications such as
the determination of Heegaard genus \cite{lackenby08-tunnel},
the recognition of small Seifert fibred spaces \cite{rubinstein04-smallsfs},
and finding bridge surfaces in knot complements \cite{wilson08-anknot}.

We begin this section by defining almost normal surfaces, whereupon we
restrict our attention to \emph{octagonal} almost normal surfaces.
Octagonal almost normal surfaces are significantly easier to deal with,
and Thompson has proven that they are sufficient for use with the 3-sphere
recognition algorithm \cite{thompson94-thinposition}.

In the remainder of this section, we define concepts similar to those
seen in Section~\ref{s-normal}, such as vector representation, matching
equations and vertex almost normal surfaces.  These concepts and their
corresponding results are well-known extensions to traditional
normal surface theory; see Lackenby \cite{lackenby08-tunnel} or
Rubinstein \cite{rubinstein97-3sphere} for a brief sketch.
The details however are not explicitly
laid down in the current literature, and so we present these details here.

\begin{definition}[Almost Normal Surface] \label{d-an}
    Let $\tri$ be a compact 3-manifold triangulation, and let
    $\Delta$ be some tetrahedron of $\tri$.  A \emph{normal octagon}
    in $\Delta$ is a properly embedded disc in $\Delta$ whose boundary
    consists of eight normal arcs running across the faces of $\Delta$,
    as illustrated in Figure~\ref{fig-anpieces}.  A \emph{normal tube}
    in $\Delta$ is a properly embedded annulus in $\Delta$ consisting of
    any two disjoint normal discs joined by an unknotted tube, again
    illustrated in Figure~\ref{fig-anpieces}.

    \begin{figure}[htb]
    \centering
    \includegraphics{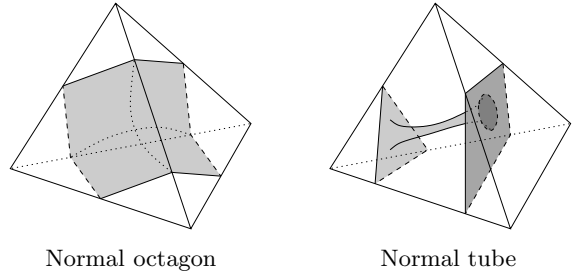}
    \caption{Examples of exceptional pieces in almost normal surfaces}
    \label{fig-anpieces}
    \end{figure}

    An \emph{almost normal surface} in $\tri$ is a properly embedded surface
    whose intersection with the tetrahedra of $\tri$ consists of
    (i)~zero or more normal discs, plus (ii)~in precisely one tetrahedron
    of $\tri$, either a single normal octagon or a single normal
    tube\footnote{Jaco and Rubinstein \cite{jaco03-0-efficiency} add the
    additional constraint that the tube does not join two copies of the same
    normal surface.}
    (but not both).  This single octagon or tube is referred to as the
    \emph{exceptional piece} of the almost normal surface.
\end{definition}

Although Definition~\ref{d-an} requires that almost normal surfaces be
properly embedded, for brevity's sake we do not include the word ``embedded''
in their name.
For the remainder of this paper we concern ourselves only with octagonal
almost normal surfaces, which are defined as follows.

\begin{definition}[Octagonal Almost Normal Surface] \label{d-octan}
    An \emph{octagonal almost normal surface} is an almost normal surface
    whose exceptional piece is a normal octagon (not a tube).  For contrast,
    we will often refer to the almost normal surfaces of
    Definition~\ref{d-an} (where the exceptional piece may be either
    an octagon or a tube) as \emph{general almost normal surfaces}.

    \begin{figure}[htb]
    \centering
    \includegraphics[scale=0.5]{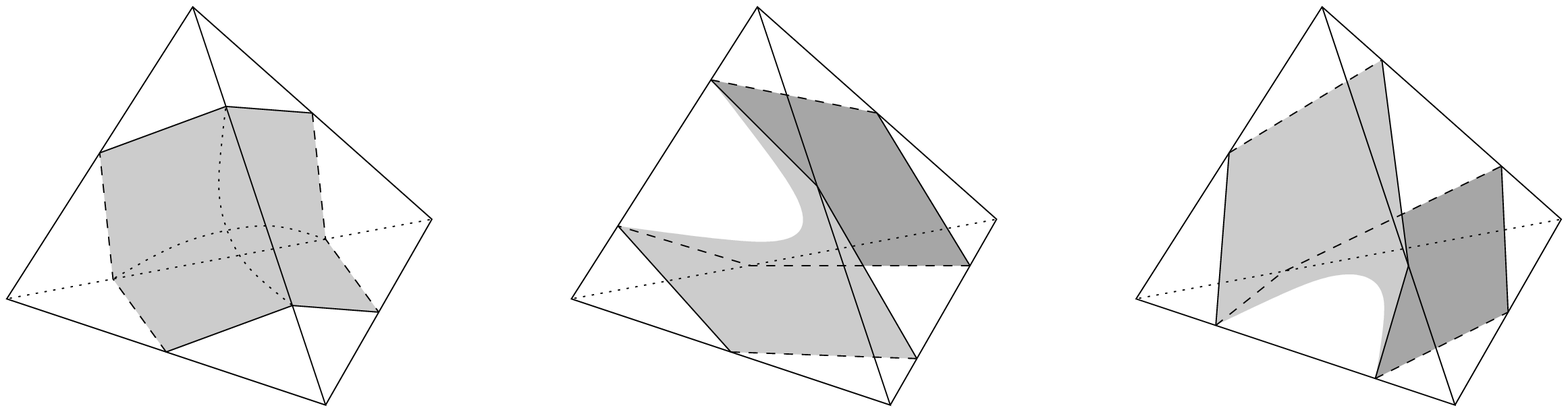}
    \caption{The three different octagon types within a tetrahedron}
    \label{fig-octtypes}
    \end{figure}

    The possible normal octagons within a tetrahedron can be grouped into three
    \emph{octagon types}, according to how many times they intersect each
    edge of the tetrahedron.  All three octagon types are illustrated
    in Figure~\ref{fig-octtypes}.
\end{definition}

As with ``embedded'', we will sometimes drop the word ``octagonal'' from
definitions to avoid excessively long names; see for instance the
\emph{standard almost normal matching equations} and
\emph{vertex almost normal surfaces}
(Definitions~\ref{d-an-vecrep} and~\ref{d-an-matching}),
which refer exclusively to octagonal almost normal surfaces.

At this early stage we can already see one reason why octagonal almost
normal surfaces are substantially easier to deal with than general almost
normal surfaces---while there are only three octagon types within a
tetrahedron, there are $25$ distinct types of normal tube, giving $28$ types
of exceptional piece in the general case.  Not only is this messier to
implement on a computer, but it can lead to significant increases
in running time and memory usage.  We return to this issue at the end of
this section.

\begin{definition}[Standard Vector Representation] \label{d-an-vecrep}
    Let $\tri$ be a compact 3-manifold triangulation formed from the $n$
    tetrahedra $\Delta_1,\ldots,\Delta_n$, and let $S$ be an octagonal
    almost normal surface in $\tri$.  For each tetrahedron $\Delta_i$, let
    $t_{i,1}$, $t_{i,2}$, $t_{i,3}$ and $t_{i,4}$ denote the number of
    triangular discs of each type, let
    $q_{i,1}$, $q_{i,2}$ and $q_{i,3}$ denote the number of
    quadrilateral discs of each type, and let
    $k_{i,1}$, $k_{i,2}$ and $k_{i,3}$ denote the number of
    octagonal discs of each type in $\Delta_i$ contained in the surface $S$.

    Then the \emph{standard vector representation} of $S$, denoted
    $\vrep{S}$, is the $10n$-dimensional vector
    \begin{alignat*}{2}
    \vrep{S}~=\,(~
        & t_{1,1},t_{1,2},t_{1,3},t_{1,4},
        \ q_{1,1},q_{1,2},q_{1,3},
        \ k_{1,1},k_{1,2}&&,k_{1,3}\ ;\\
        & t_{2,1},t_{2,2},t_{2,3},t_{2,4},
        \ q_{2,1},q_{2,2},q_{2,3},
        \ k_{2,1},k_{2,2}&&,k_{2,3}\ ; \\
        & \ldots && ,k_{n,3}\ ).
    \end{alignat*}
\end{definition}

\begin{lemma} \label{l-an-vecrep}
    Let $\tri$ be a compact 3-manifold triangulation, and let
    $S$ and $S'$ be octagonal almost normal surfaces in $\tri$.
    Then the standard vector representations
    $\vrep{S}$ and $\vrep{S'}$ are equal if and only if
    the surfaces $S$ and $S'$ are normal isotopic
    (i.e., they are the ``same'' almost normal surface).
\end{lemma}

This result is the almost normal counterpart to Lemma~\ref{l-vecrep}.  The
proof is the same, and so we do not present the details here.  The key
observation is that, given some number of triangles, quadrilaterals
and/or octagons of various types in a single tetrahedron, if these
discs can be packed into the tetrahedron disjointly
then this packing is unique up to normal isotopy.

This brings us to another reason why octagonal almost normal surfaces are
simpler to deal with than general almost normal surfaces.  In the general
case, this packing need not be unique.  In particular, a tube
that joins two normal discs of the same type can be interchanged with some
other normal disc of the same type without creating intersections
(see Figure~\ref{fig-tubeswitch} for an illustration).  Because of this,
the extension of Lemma~\ref{l-an-vecrep} to \emph{general} almost normal
surfaces fails to hold.

\begin{figure}[htb]
\centering
\includegraphics[scale=0.5]{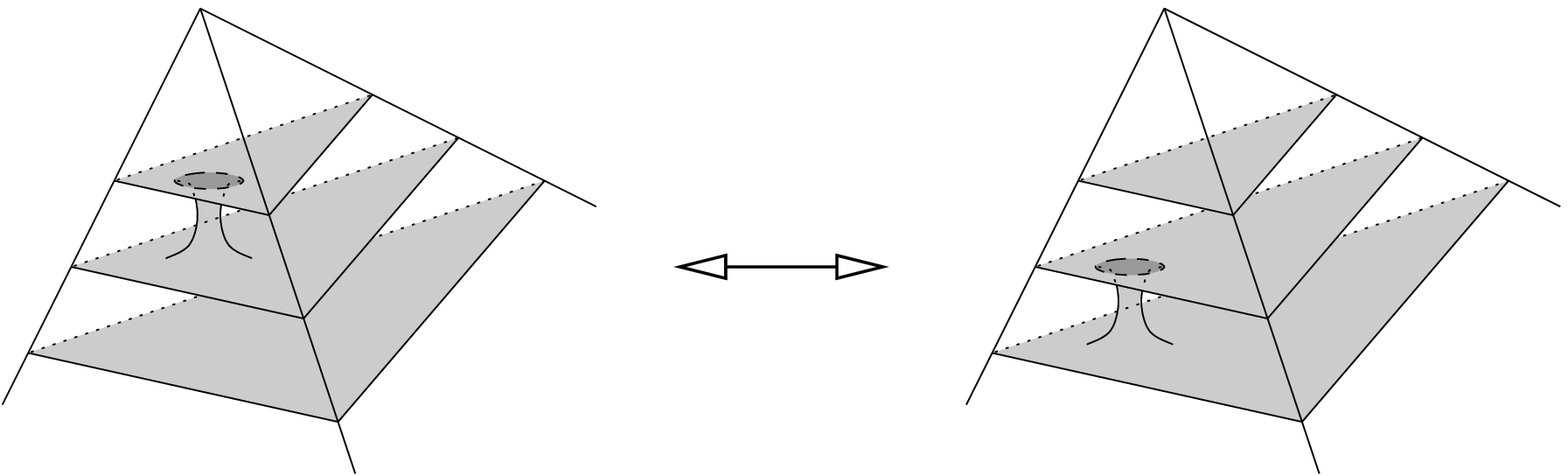}
\caption{Packing a triangle and a tube into a tetrahedron in two distinct ways}
\label{fig-tubeswitch}
\end{figure}

To determine precisely which vectors in $\R^{10n}$ represent octagonal
almost normal surfaces, we develop a set of matching equations and
quadrilateral-octagon constraints in a similar fashion to
Definitions~\ref{d-matchingstd} and~\ref{d-quadconst}.

\begin{definition}[Standard Almost Normal Matching Equations]
    \label{d-an-matching}
    Let $\tri$ be a compact 3-manifold triangulation formed from the $n$
    tetrahedra $\Delta_1,\ldots,\Delta_n$, and let $\mathbf{w} \in \R^{10n}$
    be any $10n$-dimensional vector whose coordinates we label
    \[ \mathbf{w}~=\,
    \left(\ t_{1,1},t_{1,2},t_{1,3},t_{1,4},\ q_{1,1},q_{1,2},q_{1,3},
          \ k_{1,1},k_{1,2},k_{1,3}\ ;\ \ldots,k_{n,3}\ \right). \]
    For each non-boundary face of $\tri$ and each of the three edges
    surrounding it, we obtain a
    \emph{standard almost normal matching equation} on $\mathbf{w}$ as follows.

    Let $F$ be some non-boundary face of $\tri$, and let $e$ be one
    of the three edges surrounding $F$.  Suppose that
    $\Delta_i$ and $\Delta_j$ are the two tetrahedra on either side of $F$.
    Precisely one triangular disc type, one quadrilateral
    disc type and two octagonal disc types in each of $\Delta_i$ and
    $\Delta_j$ meet $F$ in an arc parallel to $e$.
    Suppose these correspond to coordinates
    $t_{i,a}$, $q_{i,b}$, $k_{i,c}$ and $k_{i,d}$ for $\Delta_i$ and
    $t_{j,e}$, $q_{j,f}$, $k_{j,g}$ and $k_{j,h}$ for $\Delta_j$.
    Then we obtain the matching equation
    \begin{equation} \label{eqn-an-matching}
       t_{i,a} + q_{i,b} + k_{i,c} + k_{i,d} =
       t_{j,e} + q_{j,f} + k_{j,g} + k_{j,h}.
    \end{equation}
\end{definition}

These matching equations are the obvious extension to the original
standard matching equations of Definition~\ref{d-matchingstd}---we
ensure that all of the discs on one side of a non-boundary face can be
joined to all of the discs on the other side.
As with normal surfaces, if $\tri$ is a closed triangulation then there
are precisely $6n$ standard almost normal matching equations (three for each
of the $2n$ faces of $\tri$).

\begin{definition}[Quadrilateral-Octagon Constraints] \label{d-an-const}
    % XTODO: Extra space keeps the hboxes and hyphenation happy.
    ~Let $\tri$ be a compact 3-manifold triangulation formed from the $n$
    tetrahedra $\Delta_1,\ldots,\Delta_n$, and consider any vector
    \[ \mathbf{w}~=\,
    \left(\ t_{1,1},t_{1,2},t_{1,3},t_{1,4},\ q_{1,1},q_{1,2},q_{1,3},
          \ k_{1,1},k_{1,2},k_{1,3}\ ;\ \ldots,k_{n,3}\ \right) \in \R^{10n}.\]
    We say that $\mathbf{w}$ satisfies the \emph{quadrilateral-octagon
    constraints} if and only if:
    \begin{enumerate}[(i)]
        \item \label{en-quadoct-local}
        For every tetrahedron $\Delta_i$,
        at most one of the six quadrilateral and octagonal coordinates
        $q_{i,1}$, $q_{i,2}$, $q_{i,3}$,
        $k_{i,1}$, $k_{i,2}$ and $k_{i,3}$ is non-zero;
        \item \label{en-quadoct-global}
        In the entire triangulation $\tri$, at most one of the
        $3n$ octagonal coordinates $k_{1,1},\ldots,k_{n,3}$ is non-zero.
    \end{enumerate}
\end{definition}

Like the quadrilateral constraints of Definition~\ref{d-quadconst},
condition~(\ref{en-quadoct-local}) of the quadrilateral-octagon constraints
ensures that the discs within a single tetrahedron can be embedded
without intersecting.  Condition~(\ref{en-quadoct-global}) ensures that we
have at most one octagon \emph{type} within a triangulation---although this
condition is not strong enough to ensure at most one octagonal \emph{disc},
it does have the useful property of invariance under scalar multiplication.

Note that a vector can still satisfy the quadrilateral-octagon
constraints even if all its octagonal coordinates are zero.
This is necessary for the vertex enumeration algorithms to function
properly; we return to this issue in Section~\ref{s-enumeration}.

We can now give a full set of necessary and sufficient conditions for a
vector in $\R^{10n}$ to represent an octagonal almost normal surface.

\begin{theorem} \label{t-an-admissible}
    Let $\tri$ be a compact 3-manifold triangulation formed from
    $n$ tetrahedra.  An integer vector $\mathbf{w} \in \R^{10n}$
    is the standard vector representation of an
    octagonal almost normal surface in $\tri$ if and only if:
    \begin{itemize}
        \item The coordinates of $\mathbf{w}$ are all non-negative;
        \item $\mathbf{w}$ satisfies the standard almost normal
        matching equations for $\tri$;
        \item $\mathbf{w}$ satisfies the quadrilateral-octagon
        constraints for $\tri$;
        \item There is precisely one non-zero octagonal coordinate in
        $\mathbf{w}$, and this coordinate is set to one.
    \end{itemize}
    Once again, such a vector is called an \emph{admissible vector}.
\end{theorem}

Again the proof is essentially the same as for the corresponding theorem
in normal surface theory (Theorem~\ref{t-admissible}), and so we do not
reiterate the details here.  The only difference is that we now have a
global condition in the quadrilateral-octagon constraints (at most one
non-zero octagonal coordinate in the entire triangulation),
as well as an extra constraint for
admissibility (precisely one non-zero octagonal coordinate with value one).
These are to satisfy Definition~\ref{d-an}, which requires an
almost normal surface to have precisely one exceptional piece.

It is occasionally useful to consider surfaces with any number of
octagonal discs, though still at most one octagonal disc \emph{type}.
In this case the vector representation, matching equations and
quadrilateral-octagon constraints all remain the same; the only change
appears in Theorem~\ref{t-an-admissible}, where we remove the final
condition (the one that requires a unique non-zero octagonal coordinate
with a value of one).

We finish by defining a vertex almost normal surface in a similar
fashion to Definition~\ref{d-vertex}.  We are careful here to specify our
coordinate system---in Section~\ref{s-quadoct} we define a similar concept
in {\quadoct} coordinates, and (as with normal surfaces)
a vertex surface in one coordinate system need not be a vertex surface
in another.

\begin{definition}[Standard Vertex Almost Normal Surface]
    Let $\tri$ be a compact 3-ma\-ni\-fold triangulation formed from $n$
    tetrahedra.
    The \emph{standard almost normal projective solution space}
    for $\tri$ is the rational polytope formed by (i)~taking the
    polyhedral cone of all non-negative vectors in $\R^{10n}$ that satisfy
    the standard almost normal matching equations for $\tri$, and
    then (ii)~intersecting this cone with the
    hyperplane $\{\mathbf{w} \in \R^{10n}\,|\,\sum w_i = 1\}$.

    Let $S$ be an octagonal almost normal surface in $\tri$.
    If the standard vector representation $\vrep{S}$ is a
    positive multiple of some vertex of the standard almost normal projective
    solution space, then we call $S$ a
    \emph{standard vertex almost normal surface}.
\end{definition}

As with normal surfaces, we can use the enumeration of vertex almost
normal surfaces as a basis for high-level topological algorithms.
The streamlined 3-sphere recognition of Jaco, Rubinstein
and Thompson \cite{jaco03-0-efficiency} does just this---given a
``sufficiently nice'' 3-manifold triangulation $\tri$, we (i)~enumerate all
standard vertex almost normal surfaces within $\tri$, and then (ii)~search
amongst these vertex surfaces for an almost normal 2-sphere.  We return
to this algorithm in detail in Section~\ref{s-sphere}.

This suggests yet another reason to prefer octagonal almost normal surfaces
over general almost normal surfaces.
Whereas octagonal almost normal surfaces have
$10n$-dimensional vector representations, in the general case we would need
$35n$ dimensions (allowing for $25$ types of tube
in addition to the ten octagons, quadrilaterals and
triangles in each tetrahedron).  Since both the running time and memory usage
for vertex enumeration can grow exponential in the dimension of the
underlying vector space \cite{burton08-dd}, increasing this dimension
from $10n$ to $35n$ could well have a crippling effect on
performance.\footnote{We can avoid a $35n$-dimensional vertex
enumeration by exploiting the fact that every tube corresponds to a pair
of normal discs.  However, the enumeration algorithm becomes significantly
more complex as a result.}

\section{{\QuadOct} Coordinates} \label{s-quadoct}

At this stage we are ready to develop {\quadoct} coordinates, which form
the main focus of this paper.  {\Quadoct} coordinates act as an almost
normal analogy to Tollefson's quadrilateral coordinates,
in that we ``forget'' all information regarding triangular discs.  As with
quadrilateral coordinates, we happily find that---except for
vertex linking components---all of the forgotten information can be
successfully recovered.

The main results of this section are (i)~to show that vectors in {\quadoct}
coordinates uniquely identify surfaces up to vertex linking components
(Lemma~\ref{l-qo-vecrep}), and (ii)~to develop a set of necessary and
sufficient conditions for a vector in {\quadoct} coordinates to represent an
octagonal almost normal surface (Theorem~\ref{t-qo-admissible}).
Although these mirror Tollefson's original results in quadrilateral
coordinates, the proofs follow a different course---in this sense the
author hopes that this paper and Tollefson's paper
\cite{tollefson98-quadspace} make complementary reading.

\begin{definition}[{\QuadOct} Vector Representation] \label{d-qo-vecrep}
    Let $\tri$ be a compact 3-ma\-ni\-fold triangulation formed from the $n$
    tetrahedra $\Delta_1,\ldots,\Delta_n$, and let $S$ be an octagonal
    almost normal surface in $\tri$.  For each tetrahedron $\Delta_i$, let
    $q_{i,1}$, $q_{i,2}$ and $q_{i,3}$ denote the number of
    quadrilateral discs of each type, and let
    $k_{i,1}$, $k_{i,2}$ and $k_{i,3}$ denote the number of
    octagonal discs of each type in $\Delta_i$ contained in the surface $S$.

    Then the \emph{{\quadoct} vector representation} of $S$, denoted
    $\krep{S}$, is the $6n$-{\dimfix} vector
    \[ \krep{S}~=\,\left(~q_{1,1},q_{1,2},q_{1,3},
        \ k_{1,1},k_{1,2},k_{1,3}\ ;
        \ q_{2,1},q_{2,2},q_{2,3},
        \ k_{2,1},k_{2,2},k_{2,3}\ ;
        \ \ldots,k_{n,3}\ \right).
    \]
\end{definition}

Our first result in {\quadoct} coordinates is a uniqueness lemma,
analogous to Lemma~\ref{l-vecrep} for normal surfaces and
Lemma~\ref{l-an-vecrep} for standard almost normal coordinates.

\begin{lemma} \label{l-qo-vecrep}
    Let $\tri$ be a compact 3-manifold triangulation, and let
    $S$ and $S'$ be octagonal almost normal surfaces in $\tri$.
    Then the {\quadoct} vector representations
    $\krep{S}$ and $\krep{S'}$ are equal if and only if
    either (i)~the surfaces $S$ and $S'$ are normal isotopic,
    or (ii)~$S$ and $S'$ can be made normal isotopic by adding or
    removing vertex linking components.
\end{lemma}

\begin{proof}
    The ``if'' direction is straightforward.  If $S$ and $S'$ are normal
    isotopic or can be made so by adding or removing vertex linking
    components, it follows from Lemma~\ref{l-an-vecrep} that their
    \emph{standard} vector representations $\vrep{S}$ and $\vrep{S'}$
    differ only in their triangular coordinates (since
    vertex links consist entirely of triangular discs).
    Therefore the quadrilateral and octagonal coordinates
    are identical in both $\vrep{S}$ and $\vrep{S'}$,
    and we have $\krep{S}=\krep{S'}$.

    For the ``only if'' direction, suppose that $\krep{S}=\krep{S'}$.
    Let $\mathbf{d} = \vrep{S} - \vrep{S'}$ in standard almost normal
    coordinates; it follows then that
    \[ \mathbf{d}~=\,
    \left(\ t_{1,1},t_{1,2},t_{1,3},t_{1,4},\ 0,0,0,\ 0,0,0\ ;
          \ t_{2,1},t_{2,2},t_{2,3},t_{2,4},\ 0,0,0,\ 0,0,0\ ;
          \ \ldots\ \right) \in \R^{10n}\]
    for some set of triangular coordinates $\{t_{i,j}\}$.  In other words, all
    of the quadrilateral and octagonal coordinates of $\mathbf{d}$ are zero.

    We know from Theorem~\ref{t-admissible} that $\vrep{S}$ and $\vrep{S'}$
    both satisfy the standard almost normal matching equations, and because
    these equations are linear it follows that $\mathbf{d}$ satisfies them
    also.  However, with the quadrilateral and octagonal coordinates
    of $\mathbf{d}$ equal to zero, we find that each matching equation
    (\ref{eqn-an-matching}) reduces to the form
    $t_{i,a} = t_{j,e}$, where $t_{i,a}$ and $t_{j,e}$ represent
    triangular disc types surrounding a common vertex of the triangulation
    in adjacent tetrahedra (illustrated in Figure~\ref{fig-adjtri}).

    \begin{figure}[htb]
    \centering
    \includegraphics{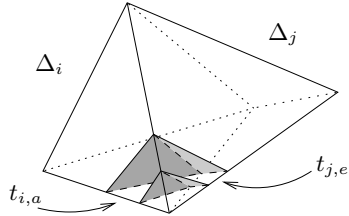}
    \caption{Adjacent triangles surrounding a common vertex}
    \label{fig-adjtri}
    \end{figure}

    By following these matching equations around each vertex of the
    triangulation $\tri$, we find that for each vertex $V$ of $\tri$, the
    coordinates $\{t_{i,j}\}$ for all triangular disc types surrounding $V$
    are equal.
    That is, $\mathbf{d} = \vrep{S} - \vrep{S'}$ is a linear combination
    of standard almost normal vector representations of vertex links.
    It follows then from Theorem~\ref{t-an-admissible} that the surfaces
    $S$ and $S'$ can be made normal isotopic only by adding or removing
    vertex linking components.\footnote{It is important to realise that
    we \emph{can} in fact add vertex linking components to an arbitrary
    surface without causing intersections.  This is possible because we can
    ``shrink'' a vertex link arbitrarily close to the vertex that it
    surrounds, allowing us to avoid any other normal or almost normal discs.}
\end{proof}

Following the pattern established in previous sections, we now turn our
attention to building a set of necessary and sufficient conditions for a
$6n$-dimensional vector to represent an almost normal surface in
{\quadoct} coordinates.  These conditions include a set of matching
equations modelled on the original quadrilateral matching equations
of Tollefson (Definition~\ref{d-qo-matching}),
and a recasting of the quadrilateral-octagon
constraints in $6n$ dimensions (Definition~\ref{d-qo-const}).
The full set of necessary and sufficient conditions is laid down and
proven in Theorem~\ref{t-qo-admissible}.

\begin{definition}[{\QuadOct} Matching Equations] \label{d-qo-matching}
    Let $\tri$ be a compact 3-ma\-ni\-fold triangulation formed from the $n$
    tetrahedra $\Delta_1,\ldots,\Delta_n$, and let $\mathbf{w} \in \R^{6n}$
    be any $6n$-dimensional vector whose coordinates we label
    \[ \mathbf{w}~=\,
        \left(\ q_{1,1},q_{1,2},q_{1,3},\ k_{1,1},k_{1,2},k_{1,3}\ ;
        \ \ldots,k_{n,3}\ \right).\]
    For each non-boundary edge of $\tri$, we obtain a
    \emph{{\quadoct} matching equation} on $\mathbf{w}$ as follows.

    Let $e$ be some non-boundary edge of $\tri$.  As with Tollefson's
    original quadrilateral matching equations, we arbitrarily label
    the two ends of $e$ as \emph{upper} and \emph{lower}.
    The tetrahedra containing edge $e$ are arranged in a cycle around $e$,
    as illustrated in the leftmost diagram of
    Figure~\ref{fig-matchingquadoct}.  Choose some arbitrary
    direction around this cycle, and suppose that the tetrahedra that we
    encounter as we travel in this direction around the cycle are labelled
    $\Delta_{i_1},\ldots,\Delta_{i_t}$.

    \begin{figure}[htb]
    \centering
    \includegraphics[scale=0.8]{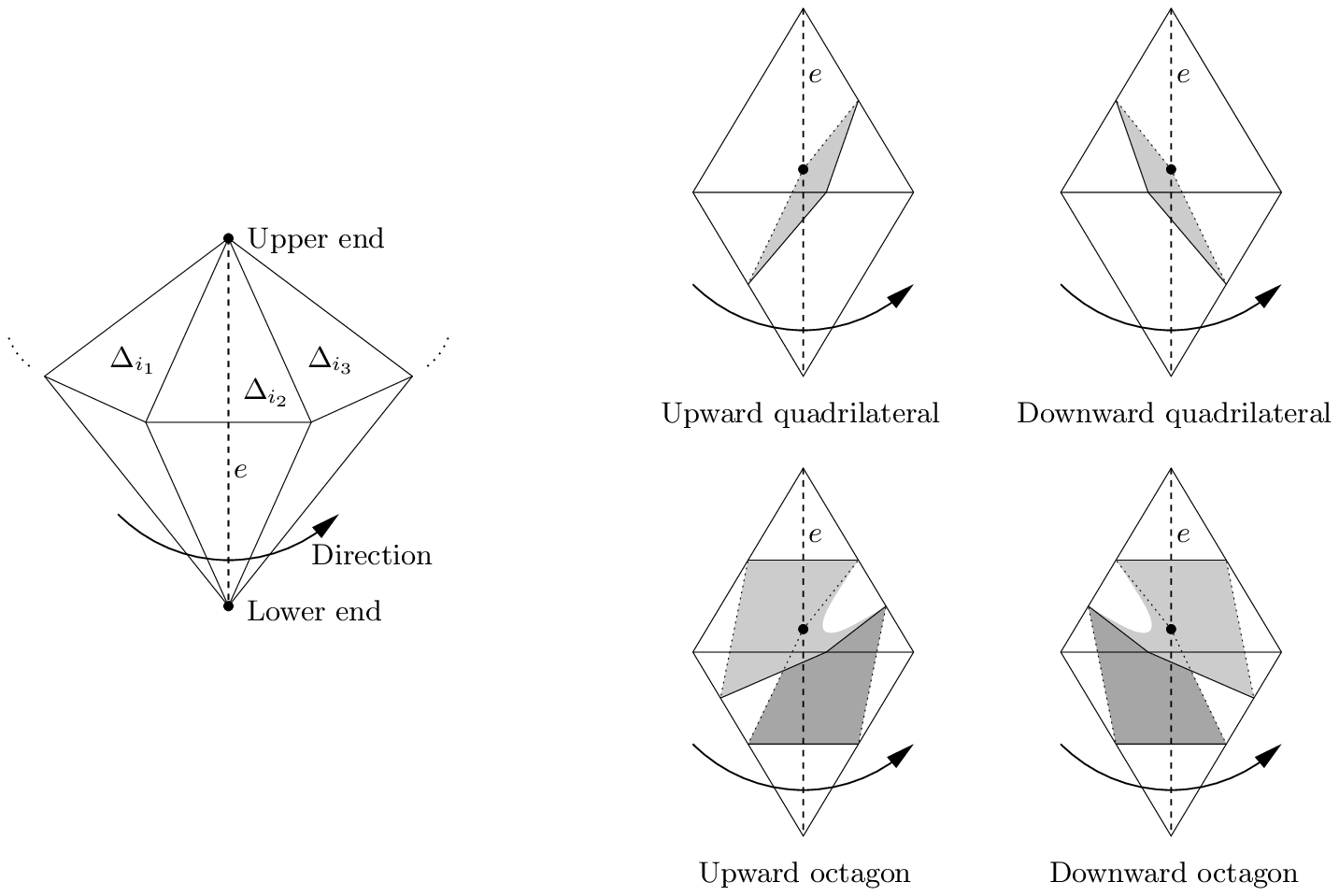}
    \caption{Building the {\quadoct} matching equations}
    \label{fig-matchingquadoct}
    \end{figure}

    Consider any tetrahedron $\Delta_{i_j}$ in this cycle.
    Within this tetrahedron, there are two quadrilateral types and two
    octagon types that meet edge $e$ precisely once.  For one
    quadrilateral and one octagon type, the intersection with $e$
    acts as a ``hinge'' about which two adjacent edges of the disc rise
    from the lower end of $e$ to the upper end of $e$ as we travel around
    the cycle in the chosen direction.  We call these disc types the
    \emph{upward quadrilateral} and the \emph{upward octagon} in
    $\Delta_{i_j}$, and we call the remaining two disc types the
    \emph{downward quadrilateral} and the \emph{downward octagon}
    in $\Delta_{i_j}$.
    All four disc types are illustrated in the rightmost portion of
    Figure~\ref{fig-matchingquadoct}.

    Suppose now that the coordinates corresponding to the upward
    quadrilateral and octagon types are
    $q_{i_1,u_1},q_{i_2,u_2},\ldots,q_{i_t,u_t}$ and
    $k_{i_1,u_1'},k_{i_2,u_2'},\ldots,k_{i_t,u_t'}$ respectively,
    and that the coordinates corresponding to the downward quadrilateral
    and octagon types are
    $q_{i_1,d_1},q_{i_2,d_2},\ldots,q_{i_t,d_t}$ and
    $k_{i_1,d_1'},\allowbreak k_{i_2,d_2'},\ldots,k_{i_t,d_t'}$
    respectively.\footnote{If we number the quadrilateral and octagon
    types within each tetrahedron in a natural way, we find that $u_j' = d_j$
    and $d_j' = u_j$ for each $j$.  That is, our numbering scheme associates
    each upward quadrilateral type with a downward octagon type and vice
    versa.  We return to this matter in Section~\ref{s-joint}.}
    Then we obtain the matching equation
    \begin{equation} \label{eqn-qo-matching}
       q_{i_1,u_1} + \ldots + q_{i_t,u_t} +
       k_{i_1,u_1'} + \ldots + k_{i_t,u_t'} =
       q_{i_1,d_1} + \ldots + q_{i_t,d_t} +
       k_{i_1,d_1'} + \ldots + k_{i_t,d_t'} .
    \end{equation}
    In other words, the total number of upward quadrilaterals and octagons
    surrounding $e$ equals the total number of downward quadrilaterals
    and octagons surrounding $e$.
\end{definition}

Note that each tetrahedron surrounding $e$ contains a third quadrilateral
type and a third octagon type, neither of which appears in
equation~(\ref{eqn-qo-matching}).  The third quadrilateral type is missing
because it does not intersect with the edge $e$ at all.  The third
octagon type is missing because, although it intersects $e$ twice, these
intersections behave in a similar fashion to two triangular discs (one
at each end of $e$).  Details can be found in the proof of
Theorem~\ref{t-qo-admissible}.

As with Tollefson's original quadrilateral matching equations, if our
triangulation $\tri$ is closed and has precisely $v$ vertices then we
obtain a total of $n+v$ {\quadoct} matching equations (one for each of the
$n+v$ edges of $\tri$).

\begin{definition}[{\QuadOct} Constraints] \label{d-qo-const}
    % XTODO: Extra space keeps the hboxes and hyphenation happy.
    ~Let $\tri$ be a compact 3-manifold triangulation formed from the $n$
    tetrahedra $\Delta_1,\ldots,\Delta_n$, and consider any vector
    \[ \mathbf{w}~=\,
        \left(\ q_{1,1},q_{1,2},q_{1,3},\ k_{1,1},k_{1,2},k_{1,3}\ ;
        \ \ldots,k_{n,3}\ \right) \in \R^{6n}.\]
    We say that $\mathbf{w}$ satisfies the \emph{quadrilateral-octagon
    constraints} if and only if:
    \begin{enumerate}[(i)]
        \item \label{en-qo-quadoct-local}
        For every tetrahedron $\Delta_i$,
        at most one of the six quadrilateral and octagonal coordinates
        $q_{i,1}$, $q_{i,2}$, $q_{i,3}$,
        $k_{i,1}$, $k_{i,2}$ and $k_{i,3}$ is non-zero;
        \item \label{en-qo-quadoct-global}
        In the entire triangulation $\tri$, at most one of the
        $3n$ octagonal coordinates $k_{1,1},\ldots,k_{n,3}$ is non-zero.
    \end{enumerate}
\end{definition}

Note that Definition~\ref{d-qo-const} is essentially a direct copy of
the quadrilateral-octagon constraints for standard almost normal coordinates
(Definition~\ref{d-an-const}), merely recast in $6n$ dimensions
instead of $10n$.

We can now describe the full set of necessary and sufficient conditions
for a vector to represent an almost normal surface in {\quadoct}
coordinates.  The resulting theorem incorporates aspects of both
Theorem~\ref{t-admissible} (which uses Tollefson's original quadrilateral
matching equations) and Theorem~\ref{t-an-admissible} (which introduces
the quadrilateral-octagon constraints).

\begin{theorem} \label{t-qo-admissible}
    Let $\tri$ be a compact 3-manifold triangulation formed from
    $n$ tetrahedra.  An integer vector $\mathbf{w} \in \R^{6n}$
    is the {\quadoct} vector representation of an
    octagonal almost normal surface in $\tri$ if and only if:
    \begin{itemize}
        \item The coordinates of $\mathbf{w}$ are all non-negative;
        \item $\mathbf{w}$ satisfies the {\quadoct}
        matching equations for $\tri$;
        \item $\mathbf{w}$ satisfies the quadrilateral-octagon
        constraints for $\tri$;
        \item There is precisely one non-zero octagonal coordinate in
        $\mathbf{w}$, and this coordinate is set to one.
    \end{itemize}
    Yet again, such a vector is called an \emph{admissible vector}.
\end{theorem}

\begin{proof}
    We begin by showing that the four conditions listed in
    Theorem~\ref{t-qo-admissible} are necessary.
    Let $S$ be some octagonal almost normal surface in $\tri$.
    It is clear from Theorem~\ref{t-an-admissible} that the
    {\quadoct} vector representation $\krep{S}$
    is a non-negative vector that satisfies the quadrilateral-octagon
    constraints, and that there is precisely one non-zero octagonal
    coordinate in $\krep{S}$ whose value is set to one.
    All that remains then is to show that $\krep{S}$ satisfies the
    {\quadoct} matching equations, which is a simple matter of combining
    the \emph{standard} almost normal matching equations appropriately.
    The details are as follows.

    Suppose that $S$ has standard vector representation
    \[ \vrep{S}~=\,
        \left(\ t_{1,1},t_{1,2},t_{1,3},t_{1,4},\ q_{1,1},q_{1,2},q_{1,3},
        \ k_{1,1},k_{1,2},k_{1,3}\ ;\ \ldots,k_{n,3}\ \right). \]
    Let $e$ be any non-boundary edge of $\tri$, and arbitrarily label
    the two ends of $e$ as \emph{upper} and \emph{lower}.  Following
    Definition~\ref{d-qo-matching}, let the tetrahedra containing $e$
    be labelled $\Delta_{i_1},\ldots,\Delta_{i_t}$ as we cycle in some
    arbitrary direction around $e$,
    let coordinates $q_{i_1,u_1},q_{i_2,u_2},\ldots,q_{i_t,u_t}$ and
    $k_{i_1,u_1'},k_{i_2,u_2'},\ldots,k_{i_t,u_t'}$ correspond to the
    upward quadrilateral and octagon types, and let coordinates
    $q_{i_1,d_1},q_{i_2,d_2},\ldots,q_{i_t,d_t}$ and
    $k_{i_1,d_1'},k_{i_2,d_2'},\ldots,k_{i_t,d_t'}$ correspond to the
    downward quadrilateral and octagon types.

    We continue labelling coordinates as follows.
    Suppose that $t_{i_1,a_1},t_{i_1,a_2},\ldots,t_{i_1,a_t}$
    correspond to the triangular disc types surrounding the upper end of
    $e$, as illustrated in the left-hand portion of Figure~\ref{fig-uppertri}.
    Furthermore, suppose that $k_{i_1,b_1},k_{i_2,b_2},\ldots,k_{i_t,b_t}$
    correspond to the octagonal disc types in each tetrahedron that are
    neither upward nor downward octagons, as illustrated in the right-hand
    portion of Figure~\ref{fig-uppertri}.

    \begin{figure}[htb]
    \centering
    \includegraphics[scale=0.8]{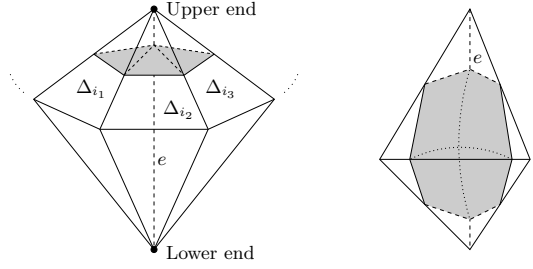}
    \caption{Triangles and octagons for the coordinates
        $t_{i_j,a_j}$ and $k_{i_j,b_j}$}
    \label{fig-uppertri}
    \end{figure}

    Calling on Theorem~\ref{t-an-admissible} again, we know that
    $\vrep{S}$ satisfies the standard almost normal matching equations
    (Definition~\ref{d-an-matching}).  Amongst those matching equations
    that involve the adjacent pairs of tetrahedra
    $(\Delta_{i_1},\Delta_{i_2})$,
    $(\Delta_{i_2},\Delta_{i_3})$, \ldots,
    $(\Delta_{i_t},\Delta_{i_1})$,
    we find the $t$ equations
    \begin{equation}\begin{split}
        t_{i_1,a_1} + q_{i_1,u_1} + k_{i_1,u_1'} + k_{i_1,b_1} &=
        t_{i_2,a_2} + q_{i_2,d_2} + k_{i_2,d_2'} + k_{i_2,b_2}, \\
        t_{i_2,a_2} + q_{i_2,u_2} + k_{i_2,u_2'} + k_{i_2,b_2} &=
        t_{i_3,a_3} + q_{i_3,d_3} + k_{i_3,d_3'} + k_{i_3,b_3}, \\
        &\ \,\vdots \\
        t_{i_t,a_t} + q_{i_t,u_t} + k_{i_t,u_t'} + k_{i_t,b_t} &=
        t_{i_1,a_1} + q_{i_1,d_1} + k_{i_1,d_1'} + k_{i_1,b_1}.
        \label{eqn-summatching}
    \end{split}\end{equation}
    Summing these together and cancelling the common terms
    $\{t_{i_j,a_j}\}$ and $\{k_{i_j,b_j}\}$, we obtain
    \[ q_{i_1,u_1} + \ldots + q_{i_t,u_t} +
       k_{i_1,u_1'} + \ldots + k_{i_t,u_t'} =
       q_{i_1,d_1} + \ldots + q_{i_t,d_t} +
       k_{i_1,d_1'} + \ldots + k_{i_t,d_t'} .\]
    That is, the {\quadoct} vector representation $\mathbf{k}(S)$
    satisfies the {\quadoct} matching equations.

    % To the compositor:  This vertical skip separates the two main
    % sections of a long proof.
    \medskip

    We now turn to the more interesting task of proving that our list of
    conditions is \emph{sufficient} for an integer vector
    $\mathbf{w} \in \R^{6n}$ to represent an octagonal almost normal surface.
    Let
    \[ \mathbf{w}~=\,
        \left(\ q_{1,1},q_{1,2},q_{1,3},
        \ k_{1,1},k_{1,2},k_{1,3}\ ;\ \ldots,k_{n,3}\ \right) \in \R^{6n} \]
    be an arbitrary integer vector that satisfies the
    four conditions listed in the statement of this theorem.
    Our aim is to extend $\mathbf{w}$ to an integer vector
    \[ \mathbf{w}'~=\,
        \left(\ t_{1,1},t_{1,2},t_{1,3},t_{1,4},\ q_{1,1},q_{1,2},q_{1,3},
        \ k_{1,1},k_{1,2},k_{1,3}\ ;\ \ldots,k_{n,3}\ \right) \in \R^{10n} \]
    that satisfies the conditions of Theorem~\ref{t-an-admissible}.
    If we can do this, it will follow from Theorem~\ref{t-an-admissible}
    that $\mathbf{w}'$ is the standard almost normal vector representation
    of some octagonal almost normal surface in $\tri$, whereupon
    $\mathbf{w}$ must be the {\quadoct} vector representation of this
    same surface.

    Given our conditions on $\mathbf{w} \in \R^{6n}$, it is clear that
    any non-negative extension $\mathbf{w}' \in \R^{10n}$ will satisfy the
    quadrilateral-octagon constraints, and will have precisely one
    non-zero octagonal coordinate whose value is set to one.  All we
    must do then is show that we can find a set of non-negative
    triangular coordinates $\{t_{i,j}\}$ that satisfy the standard
    almost normal matching equations of Definition~\ref{d-an-matching}.

    Our broad strategy is to use the vertex links of $\tri$
    as a ``canvas'' on which we write the triangular coordinates
    $t_{i,j}$, and to reformulate the matching equations as local
    constraints on this canvas.  In doing this, we show that the
    standard almost normal matching equations describe a cochain
    $\alpha \in C^1(\dualtri)$, where $\dualtri$ is the dual polygonal
    decomposition of the vertex links, and that a solution
    $\{t_{i,j}\}$ exists if and only if $\alpha$ is a coboundary.
    Using the {\quadoct} matching equations we then find that $\alpha$
    is a cocycle, whereupon the result follows from the trivial homology
    of the vertex links.  The details are as follows.

    Because $\tri$ is a compact triangulation, each of its vertex links
    is a triangulated sphere or disc, as illustrated in the left-hand
    diagram of Figure~\ref{fig-linknumbers}.  Each triangular disc type
    appears once and only once amongst the vertex links, and so we can
    write each integer $t_{i,j}$ on the corresponding vertex link triangle as
    illustrated in the right-hand diagram of Figure~\ref{fig-linknumbers}.
    This is the sense in which we use the vertex links as a ``canvas''.

    \begin{figure}[htb]
    \centering
    \includegraphics{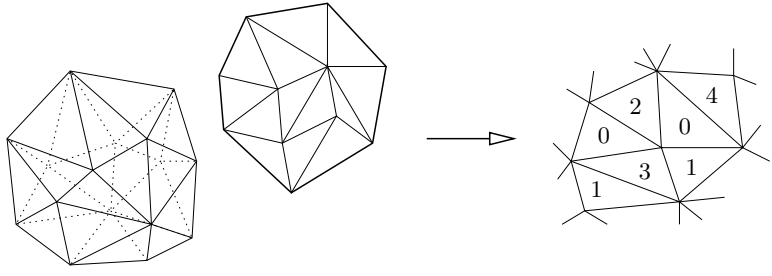}
    \caption{Writing the coordinates $t_{i,j}$ on the triangulated vertex links}
    \label{fig-linknumbers}
    \end{figure}

    We can now reformulate the standard almost normal matching equations as
    constraints on this canvas.  Recall that each standard matching
    equation involves a face $F$ of $\tri$ and arcs parallel to some edge
    $e$ of this face, as illustrated in the left-hand diagram of
    Figure~\ref{fig-linkarcs}.  We can associate every such equation
    with a single non-boundary edge $g$ of a triangulated vertex link,
    where this edge $g$ also appears as an arc of the face $F$ parallel to $e$,
    as illustrated in the right-hand diagram of Figure~\ref{fig-linkarcs}.
    In this way, the standard almost normal matching equations and the
    non-boundary edges
    of the triangulated vertex links are in one-to-one correspondence.

    \begin{figure}[htb]
    \centering
    \includegraphics{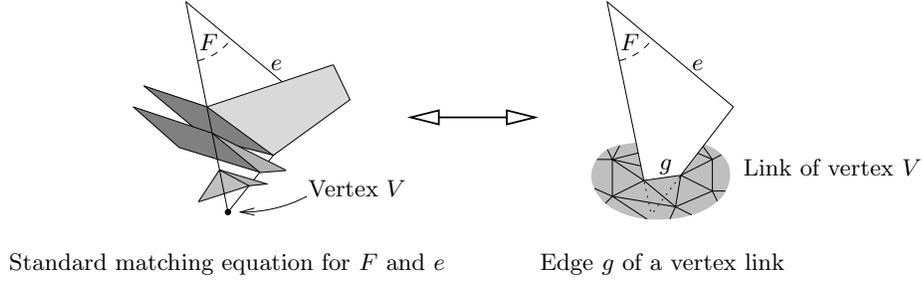}
    \caption{Associating a standard matching equation with
        an edge of a vertex link}
    \label{fig-linkarcs}
    \end{figure}

    Now consider some standard matching equation
    $t_{i,a} + q_{i,b} + k_{i,c} + k_{i,d} =
    t_{j,e} + q_{j,f} + k_{j,g} + k_{j,h}$ (as seen in
    Definition~\ref{d-an-matching}), and let $g$ be the corresponding
    edge of the triangulated vertex links.  The coordinates
    $t_{i,a}$ and $t_{j,e}$ correspond to the triangles on
    either side of $g$, and so we can write this equation in the form
    \[ t_{i,a} - t_{j,e} = K, \]
    where $K$ depends only on the quadrilateral and octagonal coordinates
    of $\mathbf{w}$.
    In other words, $K$ is a fixed quantity (dependent on the chosen
    edge $g$) that we can evaluate by looking at our original
    vector $\mathbf{w} \in \R^{6n}$.  We express this equation
    on our canvas by drawing an arrow from the triangle
    containing $t_{j,e}$ to the triangle containing $t_{i,a}$, and by
    labelling this arrow with the constant $K$.  This procedure is
    illustrated in Figure~\ref{fig-linkarrow}.

    \begin{figure}[htb]
    \centering
    \includegraphics{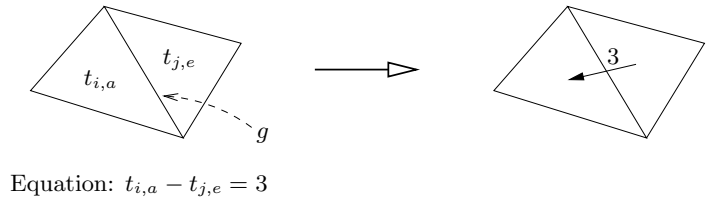}
    \caption{Representing a standard matching equation by a labelled arrow}
    \label{fig-linkarrow}
    \end{figure}

    Our situation is now as follows.  On our canvas---the triangulated
    vertex links of $\tri$---we have a labelled arrow crossing each
    non-boundary edge,
    and our task is to fill each triangle with an integer such that the
    difference across each edge matches the label on the corresponding arrow.
    An example of such a solution for a triangulated disc is illustrated in
    Figure~\ref{fig-linksoln}.  It is clear at this point that we do
    not need to worry about our non-negativity condition, since we can
    always add a constant to every triangle without changing the
    differences across the edges.

    \begin{figure}[htb]
    \centering
    \includegraphics[scale=0.8]{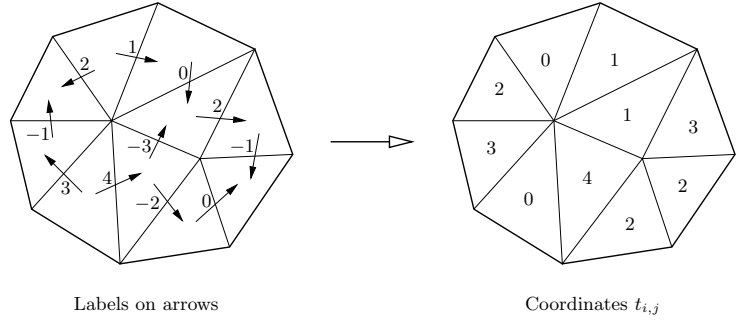}
    \caption{Solving the standard matching equations}
    \label{fig-linksoln}
    \end{figure}

    We can rephrase this using the language of cohomology.
    Let $\dualtri$ be the dual polygonal decomposition of the
    set of all vertex links, so that each triangle of a vertex link
    becomes a vertex of $\dualtri$ and each labelled arrow becomes a
    directed edge of $\dualtri$.  Then together the arrows describe a
    cochain $\alpha \in C^1(\dualtri)$ that maps each dual edge to the
    corresponding label.  A solution $\{t_{i,j}\}$
    corresponds to a cochain $\beta \in C^0(\dualtri)$ that maps each
    dual vertex to the integer in the corresponding triangle, and the
    ``difference condition'' that such a solution must satisfy is simply
    $\alpha = \cobdry \beta$.  That is,
    \emph{a solution $\{t_{i,j}\}$ exists if and only if
    $\alpha$ is a coboundary}.

    We now turn to the {\quadoct} matching equations, which we assume hold
    for our original vector $\mathbf{w} \in \R^{6n}$.  These equations
    do not involve the triangular coordinates at all.  Instead they
    tell us about the relations between different quadrilateral and
    octagonal coordinates of $\mathbf{w}$, which means they give us
    information about the labels on our arrows.

    Consider some vertex $V$ of the triangulation $\tri$, let
    $U$ be some non-boundary vertex of the triangulated link $\ell(V)$,
    and let $e$ be the edge of $\tri$ that runs through $U$ and $V$ as
    illustrated in Figure~\ref{fig-linkvertex}.
    Let $K_1,\ldots,K_t$ be the labels on the arrows surrounding $U$,
    as seen in the right-hand diagram of this figure (where we make all
    arrows point in the
    same direction around $U$ by reversing arrows and negating labels
    as necessary).  Recall that by construction, each label $K_i$ is a
    linear combination of two quadrilateral and four octagonal coordinates
    of $\mathbf{w}$.

    \begin{figure}[htb]
    \centering
    \includegraphics{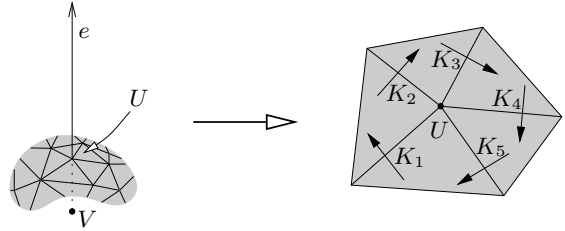}
    \caption{The triangles surrounding some vertex $U$ of the
        vertex link $\ell(V)$}
    \label{fig-linkvertex}
    \end{figure}

    Now consider the {\quadoct} matching equation constructed
    from edge $e$.  By declaring $V$ to be at the upper end of $e$, we can
    invert the procedure used earlier in equation~(\ref{eqn-summatching})
    to express our matching equation as
    \[ K_1 + \ldots + K_t = 0 .\]
    In other words, the {\quadoct} matching equations tell us that
    \emph{around every non-boundary vertex of a triangulated vertex
    link, the sum of labels on arrows is zero}.  We see this for
    instance in Figure~\ref{fig-linksoln}---by walking clockwise around
    each internal vertex and negating labels when arrows point backwards,
    the left internal vertex gives $1 + 0 - (-3) - 4 + 3 + (-1) - 2 = 0$,
    and the right internal vertex gives $2 + (-1) - 0 - (-2) + (-3) = 0$.

    Returning to our cohomology formulation, this simply tells us that
    $\cobdry \alpha = 0$, where $\alpha \in C^1(\dualtri)$ is the cochain
    described earlier.  That is, \emph{$\alpha$ is a cocycle}.
    However, because each
    vertex link is a sphere or a disc, the cohomology group $H^1(\dualtri)$
    is trivial.  Therefore $\alpha$ is also a coboundary, as required.
\end{proof}

The final step of this proof shows why we must exclude the ideal
triangulations of Thurston \cite{thurston78-lectures} from our
consideration.  In an ideal triangulation, vertices
form higher-genus cusps, whereupon the vertex links become higher-genus
surfaces with non-trivial homology.
Therefore, although the {\quadoct} matching equations
still show that $\alpha \in C^1(\dualtri)$ is a cocycle in the proof above,
we can no longer conclude from this that $\alpha$ is a coboundary and that the
solution $\{t_{i,j}\}$ exists.

To finish this section, we define a vertex surface in our new coordinate
system using the same pattern that we have employed several times already.

\begin{definition}[{\QuadOct} Vertex Almost Normal Surface] \label{d-qo-vertex}
    Let $\tri$ be a compact 3-ma\-ni\-fold triangulation,
    The \emph{{\quadoct} projective solution space}
    for $\tri$ is the rational polytope formed by (i)~taking the
    polyhedral cone of all non-negative vectors in $\R^{6n}$ that satisfy
    the {\quadoct} matching equations for $\tri$, and
    then (ii)~intersecting this cone with the
    hyperplane $\{\mathbf{w} \in \R^{6n}\,|\,\sum w_i = 1\}$.

    Let $S$ be an octagonal almost normal surface in $\tri$.
    If the {\quadoct} vector representation $\krep{S}$ is a
    positive multiple of some vertex of the {\quadoct} projective
    solution space, then we call $S$ a
    \emph{{\quadoct} vertex almost normal surface}.
\end{definition}

It should be noted that, whilst it can be shown that a
connected \emph{\quadoct} vertex almost normal surface is also a
\emph{standard} vertex almost normal surface\footnote{The proof is identical
to the corresponding result for normal surfaces; see \cite{burton09-convert}
for details.},
the converse is not necessarily true.  We address this problem
for the 3-sphere recognition algorithm in the following section by proving
that the surface we seek does indeed appear as a vertex surface in
{\quadoct} coordinates.  More generally, we describe in
Section~\ref{s-enumeration} how the conversion algorithm of
\cite{burton09-convert} can reconstruct the set of all
standard vertex almost normal surfaces, given the set of all
{\quadoct} vertex almost normal surfaces as input.

\section{3-Sphere Recognition} \label{s-sphere}

The algorithm to recognise the 3-sphere has seen a significant evolution
since it was first introduced by Rubinstein in 1992.  Rubinstein's
original algorithm \cite{rubinstein97-3sphere} involved finding a
maximal disjoint collection of embedded normal 2-spheres within a
triangulation $\tri$, slicing $\tri$ open along these 2-spheres, and
then searching for almost normal 2-spheres within the complementary
regions.  Thompson \cite{thompson94-thinposition} gave an alternate
proof of this algorithm using Gabai's concept of \emph{thin position},
and also showed that it was only necessary to consider octagonal almost
normal surfaces.

The algorithm at this stage remained extremely slow\footnote{In theory
of course, since at that stage a computer implementation did not exist.}
and fiendishly difficult to implement.  The main problems were (i)~the
need to locate and deal with many normal and almost normal surfaces
simultaneously, and (ii)~the need to locate almost normal surfaces in
complementary regions of $\tri$ containing not only tetrahedra but also
sliced and truncated \emph{pieces} of tetrahedra.  Fortunately this
algorithm was simplified enormously by Jaco and Rubinstein
\cite{jaco03-0-efficiency} using the concept of \emph{0-efficient
triangulations}, to the point where a computer implementation became
practical.  The first real implementation of 3-sphere recognition
was in the software package {\regina} \cite{burton04-regina} in 2004,
over a decade after the algorithm was first introduced.

We begin this section with a brief discussion of the theory behind the
final algorithm of Jaco and Rubinstein \cite{jaco03-0-efficiency},
followed by the algorithm itself (Algorithm~\ref{a-sphere}).  A key step
of this algorithm (and indeed its bottleneck) is an
enumeration of standard vertex almost normal surfaces.  The main result
of this section is Theorem~\ref{t-qo-sphere}, in which we show that we
can restrict our attention to \emph{\quadoct} vertex normal surfaces
instead.

As noted in the introduction, the enumeration of normal and almost
normal surfaces can grow exponentially slowly in the dimension of the
underlying vector space \cite{burton08-dd}.  By using
Theorem~\ref{t-qo-sphere} we are able to reduce this dimension from
$10n$ to $6n$, which in theory should cut down the running time
substantially.  In Section~\ref{s-performance} we test this
experimentally, where indeed we find that the speed of 3-sphere
recognition improves by orders of magnitude for the cases that we
examine.

We turn our attention now to the most recent form of the 3-sphere
recognition algorithm, as given by Jaco and Rubinstein
\cite{jaco03-0-efficiency}.  The advantages of this algorithm over its
predecessors are due to the use of 0-efficient triangulations,
which are defined as follows.

\begin{definition}[0-Efficiency]
    Let $\tri$ be a closed compact 3-manifold triangulation.
    We say that $\tri$ is \emph{0-efficient} if the only embedded
    normal 2-spheres in $\tri$ are vertex links.
\end{definition}

It turns out that 0-efficient triangulations are relatively
common, in that they exist for all closed orientable irreducible
3-manifolds except for $\R P^3$ \cite[Theorem 5.5]{jaco03-0-efficiency}.
Moreover, Jaco and Rubinstein provide a procedure for explicitly
constructing a 0-efficient triangulation of such a manifold.
More generally, Jaco and Rubinstein prove the following:

\begin{theorem} \label{t-0eff-decomp}
    Let $\tri$ be a closed compact 3-manifold triangulation
    representing some (unknown) orientable 3-manifold $M$.
    Then there is a procedure to express $M$ as a connected sum
    $M = M_1 \# \ldots \# M_t$, where each $M_i$ is either given
    by a 0-efficient triangulation $\tri_i$, or is one of the special
    spaces $S^2 \times S^1$, $\R P^3$ or the lens space $L(3,1)$.
\end{theorem}

The details of this procedure can be found in Theorems~5.9 and~5.10 of
\cite{jaco03-0-efficiency} and surrounding comments.  The key idea is to
repeatedly locate embedded normal 2-spheres and crush them, until no
such 2-spheres can be found.  Note that we might still be unable to
identify the constituent manifolds $\{M_i\}$, but with the 0-efficient
triangulations $\{\tri_i\}$ we may be better placed to learn more about them.
We do not expand further on this decomposition procedure of Jaco and
Rubinstein---although it plays a key role in the 3-sphere recognition
algorithm, our focus for this paper is on a different part of the algorithm
instead.

The core result behind Jaco and Rubinstein's version of the 3-sphere
recognition algorithm is the following theorem, which builds on earlier
work of Rubinstein and Thompson
\cite{rubinstein97-3sphere,thompson94-thinposition} by exploiting
properties of 0-efficiency.  The various components of this theorem can
be found in Proposition~5.12 of \cite{jaco03-0-efficiency} and
surrounding comments.

\begin{theorem} \label{t-std-sphere}
    Let $\tri$ be a closed compact 3-manifold triangulation
    that is orientable and {\efffix}.
    Then the following statements are equivalent:
    \begin{itemize}
        \item $\tri$ is a triangulation of the 3-sphere;
        \item $\tri$ has more than one vertex,
            or $\tri$ contains an octagonal almost normal 2-sphere;
        \item $\tri$ has more than one vertex,
            or $\tri$ contains an octagonal almost normal 2-sphere
            that is a standard vertex almost normal surface.
    \end{itemize}
\end{theorem}

Based on this result, the full 3-sphere recognition algorithm of Jaco
and Rubinstein runs as follows.

\begin{algorithm}[3-Sphere Recognition] \label{a-sphere}
    Let $\tri$ be a closed compact 3-manifold triangulation, and let
    $M$ be the 3-manifold that $\tri$ represents.  The
    following algorithm decides whether or not $M$ is the 3-sphere $S^3$:
    \begin{enumerate}
        \item \label{en-sphere-homology}
        Test whether $M$ is orientable and has trivial first
        homology.  If not, then terminate with the result $M \neq S^3$.
        \item \label{en-sphere-decompose}
        Using the procedure of Theorem~\ref{t-0eff-decomp},
        express the underlying 3-manifold
        $M$ as a connected sum decomposition $M_1 \# M_2 \# \ldots \# M_t$,
        where each $M_i$ is given by a 0-efficient triangulation $T_i$.
        If this list is empty (i.e., $t=0$), then terminate with the
        result $M = S^3$.
        \item \label{en-sphere-enumerate}
        Of the 0-efficient triangulations $\tri_1,\ldots,\tri_t$,
        ignore those with more than one vertex.  For each one-vertex
        triangulation $\tri_i$:
        \begin{enumerate}[(i)]
            \item Enumerate the
            standard vertex almost normal surfaces of $\tri_i$.
            \item Search through the resulting list of surfaces for an
            almost normal 2-sphere.  If one cannot be found then
            terminate with the result $M \neq S^3$.
        \end{enumerate}
        \item If we have not yet terminated, then every 0-efficient
        triangulation $\tri_i$ has either more than one vertex or an
        almost normal 2-sphere.  In this case we conclude that $M = S^3$.
    \end{enumerate}
\end{algorithm}

% This is a short sentence wedged between two indented lists.
% Use \noindent to make it clear that this sentence is outside both lists.
\noindent
There are some points worth noting about this algorithm:
\begin{itemize}
    \item In step~\ref{en-sphere-decompose}, we do not account for the
    special spaces $S^2 \times S^1$, $L(3,1)$ and $\R P^3$ that can arise
    in the decomposition procedure of Theorem~\ref{t-0eff-decomp}.  This is
    because the homology test in step~\ref{en-sphere-homology} prevents any
    of these special spaces from appearing.

    \item The enumeration of surfaces in step~\ref{en-sphere-enumerate}
    involves a modified double description method, which is described fully
    in \cite{burton08-dd}.  We return to the enumeration algorithm
    in Section~\ref{s-enumeration}, where we discuss it from the
    perspective of {\quadoct} coordinates.
\end{itemize}

We come now to the main result of this section, which is a {\quadoct}
analogue for the earlier Theorem~\ref{t-std-sphere}.  What we essentially
show is that, for the enumeration of vertex almost normal surfaces in
step~\ref{en-sphere-enumerate} of the algorithm above, we can work
in {\quadoct} coordinates instead of standard coordinates (in other
words, $6n$ dimensions instead of $10n$).  This is important from a
practical perspective, since experience indicates that this enumeration
step is typically the bottleneck for the entire 3-sphere recognition
algorithm.\footnote{If the manifold $M$ is a connected sum of several
high-complexity homology 3-spheres, then the decomposition procedure of
Jaco and Rubinstein becomes a greater problem for performance.  However,
it is reasonable to suggest that such cases are rare in ``ordinary''
applications.}

\begin{theorem} \label{t-qo-sphere}
    Let $\tri$ be a closed compact 3-manifold triangulation
    that is orientable and {\efffix}.
    Then the following statements are equivalent:
    \begin{itemize}
        \item $\tri$ is a triangulation of the 3-sphere;
        \item $\tri$ has more than one vertex, or $\tri$
        contains an octagonal almost normal 2-sphere
        that is a {\quadoct} vertex almost normal surface.
    \end{itemize}
\end{theorem}

\begin{proof}
    We assume that $\tri$ is a one-vertex triangulation, since otherwise
    the result follows immediately from Theorem~\ref{t-std-sphere}.
    Given this, it is clear from Theorem~\ref{t-std-sphere} that $\tri$
    triangulates the 3-sphere if and only if $\tri$ contains an octagonal
    almost normal 2-sphere.  All we need to show is that,
    if $\tri$ contains an octagonal almost normal 2-sphere, then it
    contains one as a {\quadoct} vertex almost normal surface.

    Our proof is based around an idea of Casson, used also by
    Jaco and Rubinstein to prove the corresponding claim in standard
    coordinates.  We work within a face of the projective solution space
    and show that the maximum of $\chiqo(\mathbf{u}) - O(\mathbf{u})$ occurs
    at a vertex, where $\chiqo(\cdot)$ represents Euler characteristic and
    $O(\cdot)$ is the sum of octagonal coordinates.  One complication that we
    face in {\quadoct} coordinates is that, unlike the situation in
    standard coordinates, Euler characteristic is not a linear functional.
    Nevertheless, we are able to work around this difficulty by falling back
    to convexity instead.  The details are as follows.

    Suppose that $\tri$ contains some octagonal almost normal 2-sphere $S$.
    Let $\projqo \subseteq \R^{6n}$ denote the {\quadoct} projective
    solution space (Definition~\ref{d-qo-vertex}), and let $F$ be the
    minimal-dimensional face of $\projqo$ containing the vector
    representation $\krep{S}$.  This face $F$ is the face in which we
    plan to work.

    We begin by showing that every point $\mathbf{u} \in F$ satisfies
    the quadrilateral-octagon constraints.  In contrast, suppose that some
    $\mathbf{u} \in F$ does not satisfy these constraints.  Then
    for some coordinate position $i \in \{1,\ldots,6n\}$ we must have
    $u_i > 0$ where $\krep{S}_i = 0$.  Let $H$ be the hyperplane
    $H = \{\mathbf{w} \in \R^{6n}\,|\,w_i = 0\}$; it is clear that
    $H$ is a supporting hyperplane for $\projqo$, and so $H \cap F$ is a
    sub-face of $F$ containing $\krep{S}$ but not $\mathbf{u}$,
    contradicting the minimality of $F$.

    In order to define the Euler characteristic function $\chiqo : F \to \R$,
    we must understand the relationship between standard and {\quadoct} vector
    representations.  With this in mind, we define the \emph{projection map}
    $\projmap : \R^{10n} \to \R^{6n}$ and the \emph{extension map}
    $\extmap : F \to \R^{10n}$ as follows.\footnote{These maps are the
    almost normal analogues to \emph{quadrilateral projection} and
    \emph{canonical extension}, which are defined in \cite{burton09-convert}
    for the context of embedded normal surfaces.}
    \begin{itemize}
        \item For a vector $\mathbf{v} \in \R^{10n}$, the
        projection $\projmap(\mathbf{v})$ is the vector
        $\mathbf{v}$ with all triangular coordinates removed.  That is, if
        \begin{align*}
        \mathbf{v}~&=\, \left(\ t_{1,1},t_{1,2},t_{1,3},t_{1,4},
            \ q_{1,1},q_{1,2},q_{1,3}, \ k_{1,1},k_{1,2},k_{1,3}\ ;
            \ \ldots,k_{n,3}\ \right) \in \R^{10n},\ \mbox{then} \\
        \projmap(\mathbf{v})~&=\, \left(\ q_{1,1},q_{1,2},q_{1,3},
            \ k_{1,1},k_{1,2},k_{1,3}\ ; \ \ldots,k_{n,3}\ \right) \in \R^{6n}.
        \end{align*}

        \item For a vector $\mathbf{u} \in F \subset \R^{6n}$, the
        extension $\extmap(\mathbf{u})$ is defined as follows.
        Because $F \subseteq \projqo$, we know that
        $\mathbf{u}$ satisfies the {\quadoct} matching equations.
        By the same argument used in the proof of
        Theorem~\ref{t-qo-admissible}, we can therefore solve the
        \emph{standard} almost normal matching equations to obtain values for
        the missing triangular coordinates, giving us an extension
        $\mathbf{x} \in \R^{10n}$ that satisfies the standard almost normal
        matching equations and for which $\projmap(\mathbf{x}) = \mathbf{u}$.

        By the same argument used in the proof of Lemma~\ref{l-qo-vecrep},
        this extension is unique up to multiples of vertex links.
        We therefore define $\extmap(\mathbf{u})$ to be the ``minimal''
        extension, in the sense that we subtract the largest possible
        multiple of each vertex link without allowing any coordinates to
        become negative.  In other words, every coordinate of
        $\extmap(\mathbf{u})$ is non-negative, and for every vertex
        link $\ell(V)$, the coordinate for some triangular disc type
        in $\ell(V)$ is zero.

        It is important to note that, based on the way in which we solve the
        standard almost normal matching equations, if $\mathbf{u}$ is an
        integer vector then $\extmap(\mathbf{u})$ is an integer vector also.
    \end{itemize}

    It is clear that $\projmap:\R^{10n}\to\R^{6n}$ is a linear map.
    For $\extmap:F\to\R^{10n}$ the situation is a little more complex.
    By the linearity of the matching equations, it is clear that
    \begin{equation}
        \label{eqn-extscale}
        \extmap(\lambda \mathbf{u}) = \lambda \extmap(\mathbf{u})
    \end{equation}
    for any $\lambda \geq 0$.  On the other hand, for arbitrary
    $\mathbf{u},\mathbf{w} \in F$
    we only know that $\extmap(\mathbf{u} + \mathbf{w})$ and
    $\extmap(\mathbf{u}) + \extmap(\mathbf{w})$ are related by adding or
    subtracting multiples of vertex links.  Since both
    $\extmap(\mathbf{u})$ and $\extmap(\mathbf{w})$ are non-negative
    vectors, $\extmap(\mathbf{u}+\mathbf{w})$ can only
    \emph{subtract} vertex links from their sum, yielding the non-linear
    relation
    \begin{equation}
        \label{eqn-extsum}
        \extmap(\mathbf{u} + \mathbf{w}) =
        \extmap(\mathbf{u}) + \extmap(\mathbf{w})
            - \sum \lambda_i \vrep{\ell(V_i)},
    \end{equation}
    where each $\ell(V_i)$ is a vertex linking surface and each
    $\lambda_i \geq 0$.

    We can now define our Euler characteristic function as follows.
    It is well known that Euler characteristic
    is a linear functional in \emph{standard} coordinates---for an almost
    normal surface $S$ the Euler characteristic $\chi(S)$ is a linear function
    of the coordinates $\{t_{i,j}\}$, $\{q_{i,j}\}$ and
    $\{k_{i,j}\}$,\footnote{The number of
    faces in $S$ is simply $\sum t_{i,j} + \sum q_{i,j} + \sum k_{i,j}$.
    The number of vertices in $S$ is $\sum w(e_i)$, where $w(e_i)$ is
    the number of times $S$ intersects the edge $e_i$ of $\tri$,
    and where $w(e_i)$ can be written as a linear function of
    the discs in some arbitrary tetrahedron containing $e_i$.
    Edges of $S$ are dealt with in a similar way.}
    and we simply extend this to a linear functional $\chi:\R^{10n}\to\R$.
    On our face $F \subseteq \projqo$ we then define the Euler characteristic
    function $\chiqo:F \to \R$ by
    \[ \chiqo(\mathbf{u}) = \chi(\extmap(\mathbf{u})). \]

    Although $\chiqo$ is not linear on $F$, we can observe that each
    vertex link $\ell(V_i)$ is a 2-sphere, and so $\chi(\ell(V_i)) > 0$.
    Therefore equations~(\ref{eqn-extscale}) and~(\ref{eqn-extsum}) give
    \begin{alignat}{2}
        \chiqo(\lambda \mathbf{u}) &= \lambda \chiqo(\mathbf{u}) &\quad
            &\mbox{for all $\mathbf{u} \in F$ and $\lambda \geq 0$};
            \label{eqn-g-scales} \\
        \chiqo(\mathbf{u} + \mathbf{w}) &\leq
            \chiqo(\mathbf{u}) + \chiqo(\mathbf{w}) &
            &\mbox{for all $\mathbf{u},\mathbf{w} \in F$}. \nonumber
    \end{alignat}
    That is, $\chiqo$ is a \emph{convex} function on $F$.

    We are now able to exploit an analogue of the functional that Casson
    uses in standard coordinates.  Define the function
    $g:F \to \R$ by $g(\mathbf{u}) = \chiqo(\mathbf{u})-O(\mathbf{u})$,
    where $O(\mathbf{u})$ is the sum of all octagonal coordinates in
    $\mathbf{u}$.  Since $\chiqo$ is convex and $O$ is clearly linear,
    it follows that $g$ is convex also.  Therefore the maximum of $g$ is
    achieved at a vertex of the face $F$.  Let this vertex be
    $\mathbf{m} \in F$.

    Our original almost normal 2-sphere $S$ has $g(\krep{S}) = 1$, since
    $S$ has Euler characteristic two, precisely one octagonal disc,
    and no vertex linking components.  Given that $\krep{S} \in F$, it
    follows that $g(\mathbf{m}) > 0$ also.
    Using the fact that $\projqo$ is a rational polytope, we can define
    $\mathbf{m}' \in \Z^{6n}$ to be the smallest positive
    multiple of $\mathbf{m}$ with all integer coordinates.

    Given that $F \subseteq \projqo$ and that every vector in $F$ satisfies
    the {\quadoct} constraints, it follows that the extension
    $\extmap(\mathbf{m}')$ satisfies all the conditions of admissibility
    in $\R^{10n}$
    except perhaps the requirement that the unique octagonal coordinate
    is set to one---instead we might have multiple octagonal discs, or
    we might have none at all.  We can therefore reconstruct an embedded
    surface $S'$ with standard vector representation
    $\vrep{S'}=\extmap(\mathbf{m}')$, where $S'$ is one of the following:
    \begin{itemize}
        \item \label{en-vtx-an}
        an octagonal almost normal surface;
        \item \label{en-vtx-anmany}
        like an octagonal almost normal surface but with more than one
        octagonal disc;
        \item \label{en-vtx-normal}
        an embedded normal surface with no octagonal discs at all.
    \end{itemize}

    We can show that the surface $S'$ is connected as follows.
    Suppose that $S'$ consists of distinct components $S'_1,\ldots,S'_t$
    where $t>1$.  Then in {\quadoct} coordinates we have
    $\mathbf{m}' = \krep{S'} = \sum \krep{S'_i}$,
    and since $\mathbf{m}'$ is the smallest integer multiple of a vertex
    of $\projqo$ it follows that all but one of the integer vectors
    $\krep{S'_1},\ldots,\krep{S'_t}$ must be zero.
    Therefore all but one of the components $S'_i$ are vertex links,
    which is impossible because the standard vector representation
    $\vrep{S'}$ was constructed using the extension map $\extmap$.

    From equation~\ref{eqn-g-scales} we have
    $\chiqo(\mathbf{m'}) - O(\mathbf{m'}) > 0$, and because $S'$ is
    connected it follows that
    $2 \geq \chi(S') > O(\mathbf{m'}) \geq 0$.  We must therefore be in
    one of the following situations:
    \begin{enumerate}[(i)]
        \item $\chi(S') = 2$ and $O(\mathbf{m'}) = 0$.

        In this case $S'$ is an embedded normal 2-sphere.
        Since our triangulation $\tri$ is 0-efficient, it follows that
        $S'$ is a vertex link and therefore $\krep{S'}=\mathbf{0}$,
        contradicting the fact that $\krep{S'}$ is a
        positive multiple of the vertex $\mathbf{m} \in \projqo$.

        \item $\chi(S') = 1$ and $O(\mathbf{m'}) = 0$.

        In this case $S'$ is an embedded normal projective plane.
        Since $\tri$ is orientable, $S'$ must be a one-sided surface that
        doubles to an embedded normal sphere, giving the same
        contradiction as above.

        \item $\chi(S') = 2$ and $O(\mathbf{m'}) = 1$.
        \label{en-chi-ansphere}

        In this case $S'$ has precisely one octagonal disc, and is
        therefore an octagonal almost normal 2-sphere.
    \end{enumerate}

    The only case that does not yield a contradiction is
    (\ref{en-chi-ansphere}).  Since $\krep{S'}$ is a positive multiple of
    the vertex $\mathbf{m} \in \projqo$,
    it follows that $S'$ is the {\quadoct} vertex almost normal
    2-sphere that we seek.
\end{proof}

\section{Enumeration Algorithms} \label{s-enumeration}

In this section we examine the practical issue of enumerating vertex
almost normal surfaces.  We do not go into the full details of the enumeration
algorithms, since they are intricate enough to form the subjects of
papers themselves \cite{burton08-dd,burton09-convert}.  However, we do
explain in broad terms why the algorithms used for enumerating
\emph{normal} surfaces can also be used to enumerate \emph{almost normal}
surfaces in both standard and {\quadoct} coordinates, with no
unexpected changes.

The layout of this section is as follows.  We begin in
Section~\ref{s-enumeration-direct} with the direct
enumeration algorithm, which is based on a filtered double description method.
In Section~\ref{s-enumeration-convert} we discuss the conversion algorithm
from {\quadoct} to standard coordinates, which allows us to enumerate
vertex surfaces in standard coordinates substantially faster than
through a direct enumeration.  We conclude in
Section~\ref{s-enumeration-notes} with some further notes
on the implementation and use of these algorithms.

The key observations that we make for {\quadoct} coordinates are:
\begin{enumerate}[(i)]
    \item Enumerating vertex surfaces in {\quadoct} coordinates
    is a simple matter of applying the direct enumeration algorithm
    of \cite{burton08-dd} ``out of the box'', though we cannot enforce the
    ``one and only one octagon'' constraint until the algorithm has finished.

    \item \label{en-conseq-convert}
    Likewise, we can use the conversion algorithm of
    \cite{burton09-convert} out of the box to convert
    the vertices of the {\quadoct} projective solution space into the
    vertices of the standard projective solution space, though again we must
    be careful with our use of the ``one and only one octagon'' constraint.

    \item As a consequence of (\ref{en-conseq-convert}), we can use
    {\quadoct} coordinates to substantially improve the speed of
    high-level topological algorithms, even \emph{without} specific
    results such as Theorem~\ref{t-qo-sphere} that allow us
    to focus only on {\quadoct} coordinates.
\end{enumerate}

\subsection{Direct Enumeration} \label{s-enumeration-direct}

At its core, the enumeration of vertex \emph{normal} surfaces uses a
combination of the double description method of
Motzkin et al.~\cite{motzkin53-dd} and the filtering method of Letscher.
The details can be found in \cite{burton08-dd}, but essentially the
algorithm runs as follows.

Suppose we are working in the vector space $\R^d$ with $g$ matching
equations (so for a closed one-vertex triangulation we have
$d=7n$ and $g=6n$ in standard coordinates, or $d=3n$ and $g=n+1$ in
quadrilateral coordinates).
We inductively create a series of polytopes $P_0,\ldots,P_g \subseteq \R^d$
described by their vertex sets $V_0,\ldots,V_g$ according to the
following procedure:
\begin{itemize}
    \item The polytope $P_0$ is the intersection of the non-negative
    orthant in $\R^d$ with the projective hyperplane
    $\{\mathbf{x}\in\R^d\,|\,\sum x_i=1\}$, and the corresponding vertex
    set $V_0$ consists of all unit vectors in $\R^d$.
    \item The polytope $P_i$ is created by intersecting $P_{i-1}$ with a
    hyperplane corresponding to the $i$th matching equation.  The vertex
    set $V_i$ consists of vertices $\mathbf{v} \in V_{i-1}$
    that lie inside this hyperplane, as well as combinations of pairs of
    vertices $\mathbf{u},\mathbf{v} \in V_{i-1}$ that lie on opposite sides of
    this hyperplane.
\end{itemize}
The final polytope $P_g$ is the projective solution space, and
by rescaling the vertex set $V_g$ into integer coordinates
we can reconstruct the corresponding vertex normal surfaces.

Although this procedure accounts for non-negativity and the matching equations,
we have not made use of the quadrilateral constraints.  This is where
the filtering method of Letscher comes into play.
The key idea is to enforce the quadrilateral constraints at every stage
of the double description method---specifically,
we strip all vertices from each set $V_i$ that do not satisfy the
quadrilateral constraints.
Although this means that each set $V_i$ does not give a complete
representation of the polytope $P_i$, by filtering out ``bad'' vertices at
every stage of the algorithm we can
tame the exponential explosion in the size of the
vertex sets $V_i$, improving the performance of the algorithm in
practice by a substantial amount.

It is useful to understand \emph{why} this enumeration algorithm works,
so that we can see whether it can also be used with almost normal
surfaces.  In essence, the key reasons are as follows:
\begin{itemize}
    \item The double description method of Motzkin et al.\ works because
    the projective solution space is a convex polytope, defined as the
    intersection of the non-negative orthant with the projective
    hyperplane $\sum x_i = 1$ and an additional hyperplane for each
    matching equation.

    \item The filtering method of Letscher works because the
    quadrilateral constraints satisfy the following key properties:

    \displaytext{Property A}{The quadrilateral constraints are satisfied
    on a union of faces of the non-negative orthant, and therefore on a
    union of faces of the projective solution space.}

    \displaytext{Property B}{Let $\mathbf{u}$ and $\mathbf{v}$ be
        non-negative vectors in $\R^d$.  If either $\mathbf{u}$ or $\mathbf{v}$
        does not satisfy the quadrilateral constraints, then the combination
        $\alpha \mathbf{u} + \beta \mathbf{v}$ can never satisfy the
        quadrilateral constraints for any $\alpha,\beta > 0$.}

    Note that property~B is an immediate consequence of property~A,
    and that property~A holds because each constraint is of the form
    ``at most one of the coordinates $\{x_i\,|\,i \in C\}$ may be non-zero'',
    where $C \subseteq \{1,\ldots,d\}$ is some set of coordinate positions.
\end{itemize}

We now turn our attention to the enumeration of vertex
\emph{almost normal} surfaces, in both standard almost normal
coordinates and {\quadoct} coordinates.
\begin{itemize}
    \item Once again, the projective solution space is the intersection
    of the non-negative orthant with the projective hyperplane
    $\sum x_i = 1$ and an additional hyperplane for each matching
    equation.  As a result, the double description method of
    Motzkin et al.\ works seamlessly with almost normal surfaces.

    \item Like the original quadrilateral constraints,
    the quadrilateral-octagonal constraints for almost normal surfaces
    are each of the form
    ``at most one of the coordinates $\{x_i\,|\,i \in C\}$ may be non-zero'',
    where $C \subseteq \{1,\ldots,d\}$ is some set of coordinate positions.
    As a result, both of the above
    properties~A and~B hold, and we can seamlessly use
    the filtering method of Letscher to enforce the quadrilateral-octagon
    constraints at each stage of the double description method.
\end{itemize}

However, Theorems~\ref{t-an-admissible} and~\ref{t-qo-admissible} show
that octagonal almost normal surfaces come with an additional constraint:

\displaytext{Constraint ($\star$)}{For $\mathbf{v}$ to be the
    vector representation of an octagonal almost normal surface,
    there must be some non-zero octagonal coordinate in $\mathbf{v}$,
    and this coordinate must be set to one.}

It is clear that we cannot enforce ($\star$) on the projective
solution space, since there the coordinates of each vector are rationals
(not integers) that \emph{sum} to one.  From the viewpoint of the
projective solution space, this constraint is not so much a property of
a vector $\mathbf{v}$, but rather a property of the smallest multiple
of $\mathbf{v}$ with integer coordinates.  It follows that the
final constraint ($\star$) cannot be inserted verbatim into the
filtering process.

We might instead consider enforcing a weaker version of ($\star$),
where every vector $\mathbf{v} \in V_i$ must have some non-zero octagonal
coordinate (therefore eliminating vectors that yield no
octagons at all).  However, this variant is also unsuitable for
filtering, since it satisfies neither of the properties~A or~B.
In essence, the reason we must keep track of \emph{normal} surfaces (with no
octagons) is so that we can combine them with old almost normal surfaces to
create new almost normal surfaces.

The conclusion then is that we must forget the final condition ($\star$)
while the algorithm is running, and enforce it only once we have our
final set of vertices $V_g$.  Note that this is not a severe
penalty---the quadrilateral-octagon constraints already ensure that we
have at most one octagon \emph{type} in each vector, and so our only
inefficiency is that we must carry around vectors that yield too many
octagons of a single type, or that yield no octagons at all.

As a final note, the paper \cite{burton08-dd} offers a number of
additional optimisations to the core filtered double description method.
As with the core algorithm, these optimisations can also be used seamlessly
with octagonal almost normal surfaces, as long as we remember to
delay the constraint ($\star$) until after the algorithm has finished.

\subsection{The Conversion Algorithm} \label{s-enumeration-convert}

The paper \cite{burton09-convert} describes a conversion algorithm from
quadrilateral to standard coordinates for normal surfaces.  The purpose
of this algorithm is not just to convert \emph{vectors} between
coordinate systems (which is fairly straightforward), but to convert
entire \emph{solution sets}.  That is, the algorithm begins with the set
of all vertices of the \emph{quadrilateral} projective solution space
that satisfy the quadrilateral constraints, and converts this to the
(typically much larger) set of all vertices of the \emph{standard}
projective solution space that satisfy the quadrilateral constraints.
We are therefore able to recover the standard vertex normal surfaces that are
``lost'' in quadrilateral coordinates.

As a result, this algorithm allows us to enumerate all
standard vertex normal surfaces using the following two-step procedure:
\begin{enumerate}
    \item Use direct enumeration (as described in
    Section~\ref{s-enumeration-direct}) to enumerate all vertices
    of the \emph{quadrilateral} projective solution space that satisfy the
    quadrilateral constraints.
    \item Use the conversion algorithm (as described below) to recover all
    vertices of the \emph{standard} projective solution space that satisfy the
    quadrilateral constraints, and thereby the set of all standard
    vertex normal surfaces.
\end{enumerate}
Experimentation shows the conversion algorithm to have negligible running
time, and as a result this two-step procedure is found to be orders of
magnitude faster than a direct enumeration in standard coordinates
\cite{burton09-convert}.
The overall outcome is that we can harness the \emph{speed} of quadrilateral
coordinates without the need to prove additional \emph{theorems} in
quadrilateral coordinates (such as we do here for {\quadoct} coordinates
in Theorem~\ref{t-qo-sphere}).

Broadly speaking, the conversion algorithm operates as follows.
Suppose the triangulation $\tri$ is formed from $n$ tetrahedra,
and contains the $m$ vertices $V_1,\ldots,V_m$.
We inductively construct lists of vectors
$L_0,\ldots,L_m \subset \R^{7n}$ according to the following procedure:
\begin{itemize}
    \item The list $L_0$ contains the input for the algorithm, which consists
    of all vertices of the quadrilateral projective solution space that
    satisfy the quadrilateral constraints.  Each vector is extended from
    $\R^{3n}$ to $\R^{7n}$ by solving the standard matching equations.
    \item Each subsequent list $L_i$ generates all non-negative vectors
    in $\R^{7n}$ that satisfy the quadrilateral constraints, and that
    can be formed by (i)~combining vectors from the
    previous list $L_{i-1}$ and then (ii)~adding or subtracting a multiple
    of the vertex linking vector $\vrep{\ell(V_i)}$.
    This list $L_i$ is constructed from $L_{i-1}$ using an algorithm
    similar to the filtered double description method of
    Section~\ref{s-enumeration-direct}, though there are additional
    complications.
\end{itemize}
The final list $L_m$ becomes the set of all vertices of the standard
projective solution space that satisfy the quadrilateral constraints.

The key reason \emph{why} the conversion algorithm works (in addition to
those reasons discussed earlier in Section~\ref{s-enumeration-direct})
is because of the following relationship between
standard and quadrilateral coordinates:

\displaytext{Property C}{The projection from standard to
quadrilateral coordinates (where we simply remove the triangular
coordinates $\{t_{i,j}\}$) is a linear map from the standard projective
solution space to the quadrilateral projective solution space.
Moreover, the kernel of this map is generated by the standard vector
representations of the vertex links.}

We can now see why the conversion algorithm works seamlessly for almost
normal surfaces.  If we replace standard and quadrilateral
\emph{normal} coordinates with standard and {\quadoct} \emph{almost
normal} coordinates, the critical property~C still holds.  We can
thereby follow through the algorithm and its proof as presented in
\cite{burton09-convert}, and we find that the algorithm works as
expected.

Specifically, what this algorithm achieves for almost normal surfaces
is to begin with the set of all vertices of the \emph{\quadoct}
projective solution space that satisfy the quadrilateral-octagon
constraints, and to convert this to the
(again typically much larger) set of all vertices of the
\emph{standard almost normal} projective solution space that satisfy the
quadrilateral-octagon constraints.

As with direct enumeration, there is a catch involving the constraint
($\star$), which we recall insists that each vector contain a non-zero
octagonal coordinate whose value is set to one.  For the same reasons
as discussed in Section~\ref{s-enumeration-direct}, we cannot enforce
the constraint ($\star$) at each stage of the conversion algorithm.
More importantly, \emph{we cannot enforce ($\star$) on the set of input
vectors}---the input must be the set of \emph{all} vertices of the
{\quadoct} solution space that satisfy the quadrilateral-octagon
constraints, whether these vertices yield many octagonal discs or whether they
yield none.  Once again, we must delay the enforcement of ($\star$)
until the entire algorithm has finished running and we are ready to
present our final results.

As a final note, we observe that the conversion algorithm allows us to
enumerate all standard vertex almost normal surfaces using the following
two-step procedure:
\begin{enumerate}
    \item Use direct enumeration to enumerate all vertices
    of the \emph{\quadoct} projective solution space that satisfy the
    quadrilateral-octagon constraints, taking care not to
    enforce the extra constraint ($\star$).
    \item Use the conversion algorithm to recover all vertices
    of the \emph{standard almost normal} projective solution space that
    satisfy the quadrilateral-octagon constraints, and thereby the set of all
    standard vertex almost normal surfaces.
\end{enumerate}
As is the case with normal surfaces, experimentation shows that this
two-step procedure runs orders of
magnitude faster than a direct enumeration in standard coordinates.

\subsection{Further Notes} \label{s-enumeration-notes}

We finish with some additional notes on the implementation and use
of the enumeration and conversion algorithms.

Our first observation is the following.
Although we work in $10n$ and $6n$ dimensions for standard
almost normal and {\quadoct} coordinates respectively, these large
dimensions seem wasteful.  The quadrilateral-octagon constraints guarantee
at most one non-zero octagonal coordinate for each vector, so
a different possibility might be to ``select'' a desired octagonal disc
type and then work in $7n+1$ or $3n+1$ dimensions instead.

Casson has suggested such a technique \cite{jaco02-algorithms-essential},
where we iterate through all $3n$ possible octagonal disc types,
and for each such disc type we augment a traditional coordinate
system for \emph{normal} surfaces with a single
coordinate for this octagon.  As a result we obtain $3n$ distinct
projective solution spaces, each with the significantly smaller
dimension $7n+1$ or $3n+1$.

Although this reduction in dimensions is appealing, in practice both
procedures essentially perform the
same computations---by working in a full set of standard almost normal
or {\quadoct} coordinates, we are simply performing the $3n$
smaller enumerations of Casson ``simultaneously''.  This is
because the quadrilateral-octagon constraints enforce at most one
non-zero octagonal coordinate, and so the set of vertices at each
stage of the enumeration algorithm is essentially the union of
all $3n$ vertex sets in Casson's scheme, with no additional ``junk''
vertices that must later be thrown away.

More importantly however, any
enumeration of vertex \emph{almost normal} surfaces includes an
implicit enumeration
of vertex \emph{normal} surfaces, since the {\quadfix}-octagon
constraints allow surfaces with no octagons at all.  To this end, a
single ``simultaneous'' enumeration in $10n$ or $6n$ dimensions should
be more efficient---if we run $3n$ independent enumerations in $3n$
different projective solution spaces, then we effectively perform
this implicit (and potentially slow \cite{burton09-convert})
normal surface enumeration $3n$ distinct times.

Our second observation involves the constraint ($\star$) from
Section~\ref{s-enumeration-direct}---recall that this is the final
condition of Theorems~\ref{t-an-admissible} and~\ref{t-qo-admissible},
where we insist that there is some non-zero octagonal coordinate, and
that this coordinate is set to one.  We have already observed that
($\star$) cannot be enforced during either the enumeration or conversion
algorithms, and that we must instead apply it as a filter after the
algorithms have finished.

It is worth noting that there are situations in which we do not want to
enforce ($\star$) at all, even after the algorithms have run.  We have
already seen one example in Section~\ref{s-enumeration-convert},
where the conversion algorithm requires that we do not enforce ($\star$)
on the vertices in {\quadoct} coordinates.  Another example arises in
applications where we use the \emph{vertex} almost normal surfaces as
a basis to generate \emph{all} almost normal surfaces (possibly with
some limitations such as genus to keep the list finite).\footnote{Such
applications do appear in the literature; see \cite{lackenby08-tunnel}
and \cite{rubinstein04-smallsfs} for examples.}

In this case we cannot enforce ($\star$) either, since it is possible
to obtain new admissible vectors through combinations of old vectors that
break ($\star$).  For instance, we could combine an almost normal
surface with a plain normal surface (having no octagonal discs) to obtain a
new almost normal surface, or we could combine a surface with two octagons
with a plain normal surface to obtain the double of a new
almost normal surface, whereupon we simply divide by two.

\section{Measuring Performance} \label{s-performance}

In this section we measure the practical benefits of using {\quadoct}
coordinates.  We do this by experimentally comparing running times for
the 3-sphere recognition algorithm, using different coordinate systems
for the critical step in which we enumerate vertex almost normal surfaces.

For our experiments we use the 15 smallest-volume homology 3-spheres
from the closed hyperbolic census of Hodgson and Weeks
\cite{hodgson94-closedhypcensus}.  The reason for choosing homology
3-spheres is because we want to focus on almost normal surface
enumeration---manifolds with non-trivial homology are eliminated in
the first step of the 3-sphere recognition algorithm, and experience suggests
that most \emph{real} 3-spheres simplify to trivially small pieces
during the decomposition
procedure in the second step of the algorithm.\footnote{It is, however,
possible to construct arbitrarily large 0-efficient triangulations of
the 3-sphere \cite{jaco03-0-efficiency}.}

We use 0-efficient triangulations of these homology 3-spheres,
with sizes ranging from 10 to 14 tetrahedra.
Table~\ref{tab-exptspheres} lists the volume of each manifold,
the size of each triangulation, and the Dehn filling given by
Hodgson and Weeks to reconstruct each manifold.  Each Dehn filling is
applied to a cusped manifold from the hyperbolic census of
Hildebrand and Weeks \cite{hildebrand89-cuspedcensusold}.

\begin{table}[htb]
\centering
\begin{tabular}{c|c|c}
\emph{Hyperbolic volume} & \emph{Dehn filling} & \emph{Tetrahedra} \\
\hline
1.39850888 & $\mathrm{m004}(\phantom{-}1, 2)$ & 10 \\
1.91221025 & $\mathrm{m011}(\phantom{-}2, 3)$ & 11 \\
2.22671790 & $\mathrm{m015}(-3, 2)$ & 11 \\
2.25976713 & $\mathrm{m038}(\phantom{-}1, 2)$ & 11 \\
2.51622138 & $\mathrm{m081}(\phantom{-}3, 2)$ & 12 \\
2.62940540 & $\mathrm{m032}(\phantom{-}5, 2)$ & 12 \\
2.71245881 & $\mathrm{m120}(-3, 2)$ & 12 \\
2.86563023 & $\mathrm{m137}(-5, 1)$ & 13 \\
2.98683705 & $\mathrm{m137}(\phantom{-}5, 1)$ & 13 \\
3.08052001 & $\mathrm{m154}(-2, 3)$ & 12 \\
3.08386105 & $\mathrm{m137}(-6, 1)$ & 14 \\
3.16236729 & $\mathrm{m137}(\phantom{-}6, 1)$ & 14 \\
3.40043687 & $\mathrm{m222}(-3, 2)$ & 13 \\
3.44586464 & $\mathrm{m199}(-5, 1)$ & 14 \\
3.54091542 & $\mathrm{m260}(-3, 2)$ & 13
\end{tabular}
\caption{The 15 homology 3-spheres used for experimentation}
\label{tab-exptspheres}
\end{table}

For each of our 15 triangulations, we compare the running times for the
following two procedures:
\begin{itemize}
    \item 3-sphere recognition as given in Algorithm~\ref{a-sphere},
    using standard almost normal coordinates for the vertex enumeration in
    step~\ref{en-sphere-enumerate} of the algorithm;
    \item The same algorithm, but using {\quadoct} coordinates for the
    vertex enumeration in step~\ref{en-sphere-enumerate}, as authorised
    by Theorem~\ref{t-qo-sphere}.
\end{itemize}
All experiments were performed on a single 2.3\,GHz AMD Opteron
processor using the software package {\regina} \cite{regina,burton04-regina}.

The running times are plotted in Figure~\ref{fig-times} using
log scales, and the results are extremely pleasing.
Even in the worst case, {\quadoct} coordinates still
improve the running time by a factor of 30.  At the other extreme,
for several triangulations we find that {\quadoct} coordinates improve
the running time by factors of thousands, with an increase of just under
$5000$ times the speed for the best example.

\begin{figure}[htb]
\centering
\includegraphics[scale=0.6]{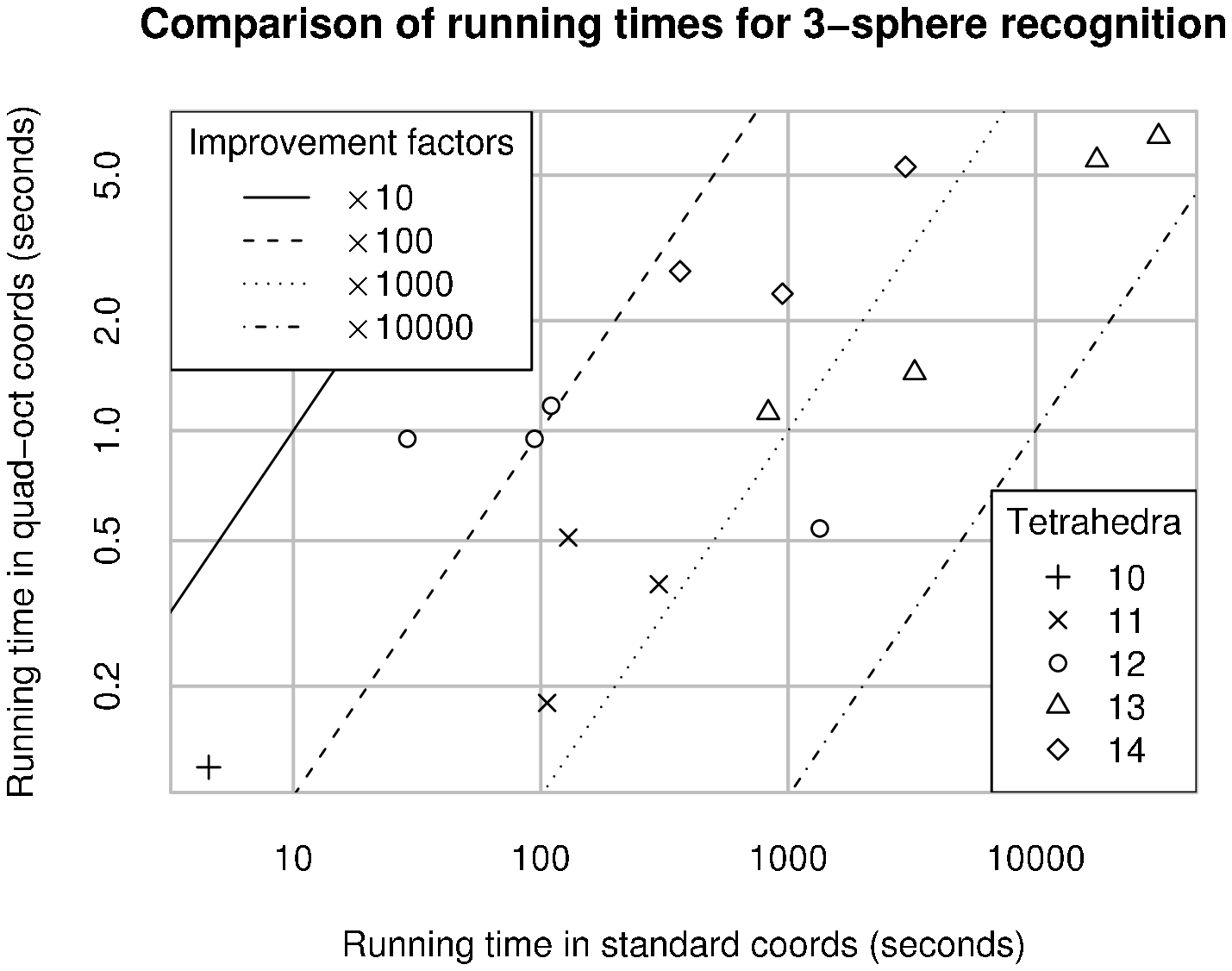}
\caption{Performance comparisons for 3-sphere recognition}
\label{fig-times}
\end{figure}

\section{Joint Coordinates} \label{s-joint}

We finish this paper with an exploratory discussion of
\emph{joint coordinates} for octagonal almost normal surfaces.
Where {\quadoct} coordinates reduce the dimension of the underlying
vector space from $10n$ to $6n$, joint coordinates reduce this even
further from $6n$ to $3n$.  The key idea is to use negative coordinates
for octagons and positive coordinates for quadrilaterals, noting from
the quadrilateral-octagon constraints that the two cannot occur together
within the same tetrahedron.

Joint coordinates have a number of appealing properties.  Not only is
their dimension small, but they carry the same information as {\quadoct}
coordinates (in contrast to the step from standard to {\quadoct} coordinates,
where we lose information about vertex links).  Moreover, joint coordinates
adhere to almost the same constraints in $\R^{3n}$ as Tollefson's quadrilateral
coordinates for \emph{normal} surfaces---in particular, they satisfy the
original quadrilateral matching equations and quadrilateral constraints
from Section~\ref{s-normal}.

There is a cost however, which is the loss of convexity.  For joint
coordinates, we must allow one coordinate (but no more) to become negative.
As a result we no longer work in the non-negative orthant of
$\R^{3n}$, but rather the non-negative orthant \emph{and} the $3n$
``almost non-negative'' orthants that border it.  This has severe
consequences for the enumeration algorithms described in
Section~\ref{s-enumeration}, which rely on convexity as a core requirement.

Nevertheless, it is pleasing to be able to express octagonal almost normal
surfaces using essentially the same coordinate system as normal surfaces,
and to do so in a way that portrays them as a natural extension of Tollefson's
original framework (where our extension involves simply stepping
``just outside'' the non-negative orthant).

The layout of this section is as follows.
We begin by describing the way in which we number quadrilateral and octagon
types within each tetrahedron, which must be done carefully for joint
coordinates to work.  Following this, we define
joint coordinates and develop the corresponding uniqueness and
admissibility results.  We then present an example using a
one-tetrahedron triangulation, where we show graphically how the vector
representations of normal and almost normal surfaces appear in the
corresponding solution space in $\R^3$.  To finish, we discuss
how the loss of convexity affects both the projective solution space and
the enumeration algorithms.

\begin{definition}[Quadrilateral and Octagon Numbering] \label{d-octquadnumber}
    Let $\Delta$ be any tetrahedron in some compact 3-manifold triangulation.
    Within $\Delta$, we number the quadrilateral and octagon types $1$,
    $2$ and $3$ so that, for each $i \in \{1,2,3\}$,
    the two edges of $\Delta$ that quadrilaterals of type $i$ \emph{never} meet
    are the same two edges of $\Delta$ that octagons of type $i$ meet
    \emph{twice}.  This correspondence between quadrilaterals and octagons
    is illustrated in Figure~\ref{fig-octquadtypes}.

    \begin{figure}[htb]
    \centering
    \includegraphics[scale=0.4]{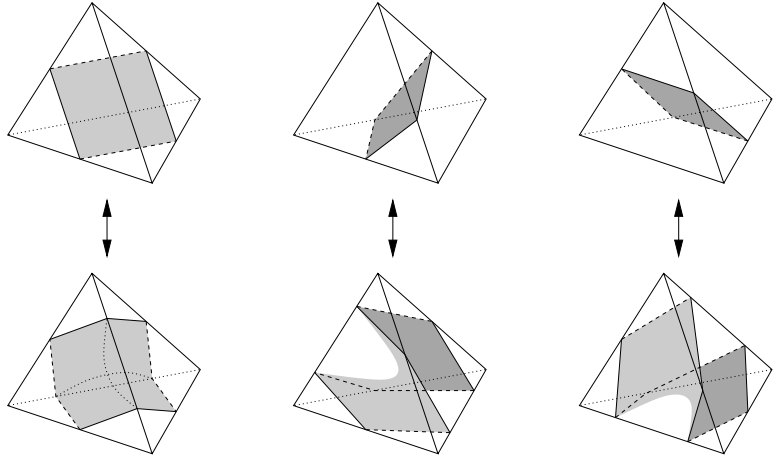}
    \caption{The correspondence between quadrilaterals and octagons}
    \label{fig-octquadtypes}
    \end{figure}
\end{definition}

This numbering scheme is very natural, in that the correspondence between
quadrilaterals and octagons reflects the natural symmetries of these
discs within a tetrahedron.  One can also think of an octagon of
type~$i$ as being obtained from a quadrilateral of type~$i$ by pulling
each edge of the quadrilateral around and over the nearby vertex of the
tetrahedron.

With our numbering scheme in place, we are now ready to define joint
coordinates.  Since we plan to present joint coordinates as a
natural extension of Tollefson's quadrilateral coordinates, we define
them for \emph{both} normal and almost normal surfaces as follows.

\begin{definition}[Joint Vector Representation] \label{d-joint-vecrep}
    Let $\tri$ be a compact 3-manifold triangulation formed from the $n$
    tetrahedra $\Delta_1,\ldots,\Delta_n$, and let $S$ be an embedded normal
    or octagonal almost normal surface in $\tri$.  For each tetrahedron
    $\Delta_i$, let $q_{i,1}$, $q_{i,2}$ and $q_{i,3}$ denote the number of
    quadrilateral discs of each type, and let
    $k_{i,1}$, $k_{i,2}$ and $k_{i,3}$ denote the number of
    octagonal discs of each type in $\Delta_i$ contained in the surface $S$.

    Then the \emph{joint vector representation} of $S$, denoted
    $\jrep{S}$, is the $3n$-{\dimfix} vector
    \begin{alignat*}{2}
    \jrep{S}~=\,(~
        & q_{1,1}-k_{1,1},\,q_{1,2}-k_{1,2}&&,\,q_{1,3}-k_{1,3}\ ;\\
        & q_{2,1}-k_{2,1},\,q_{2,2}-k_{2,2}&&,\,q_{2,3}-k_{2,3}\ ;\\
        & \ldots                           &&,\,q_{n,3}-k_{n,3}\ ).
    \end{alignat*}
\end{definition}

Our first task is to show that joint coordinates in $3n$ dimensions retain
all of the information carried by {\quadoct} coordinates in $6n$ dimensions.

\begin{lemma} \label{l-joint-vecrep}
    Let $\tri$ be a compact 3-manifold triangulation, and let
    $S$ and $S'$ be embedded normal or octagonal almost normal surfaces in
    $\tri$.  Then the joint vector representations $\jrep{S}$ and $\jrep{S'}$
    are equal if and only if the {\quadoct} vector representations
    $\krep{S}$ and $\krep{S'}$ are equal.
\end{lemma}

Here we define the {\quadoct} vector representation for an embedded
normal surface in the obvious way, by setting all octagonal
coordinates to zero.

\begin{proof}
    It is clear that if $\krep{S}=\krep{S'}$ then $\jrep{S}=\jrep{S'}$.
    Suppose conversely that $\jrep{S}=\jrep{S'}$, and consider the
    $(i,t)$th coordinate $j_{i,t} = q_{i,t} - k_{i,t}$.

    For both $S$ and $S'$ we know that $q_{i,t} \geq 0$ and $k_{i,t} \geq 0$.
    Moreover, since $S$ and $S'$ satisfy the quadrilateral-octagon
    constraints, we know that they can each have $q_{i,t} > 0$ or
    $k_{i,t} > 0$ but not both.  It follows that for both $S$ and $S'$ we
    have one of the following situations:
    \begin{itemize}
        \item $j_{i,t} = 0$, in which case $q_{i,t} = k_{i,t} = 0$;
        \item $j_{i,t} = K > 0$, in which case $q_{i,t} = K$ and $k_{i,t} = 0$;
        \item $j_{i,t} = -K < 0$, in which case $q_{i,t} = 0$ and $k_{i,t} = K$.
    \end{itemize}
    That is, we can reconstruct the individual constituents $q_{i,t}$
    and $k_{i,t}$ from the joint coordinate $j_{i,t}$, whereupon we
    obtain $\krep{S}=\krep{S'}$.
\end{proof}

As an immediate consequence of Lemmas~\ref{l-joint-vecrep}
and~\ref{l-qo-vecrep}, we obtain the following uniqueness result for
joint vector representations:

\begin{corollary} \label{c-joint-vecrep}
    Let $\tri$ be a compact 3-manifold triangulation, and let
    $S$ and $S'$ be embedded normal or octagonal almost normal surfaces
    in $\tri$.  Then the joint vector representations
    $\jrep{S}$ and $\jrep{S'}$ are equal if and only if
    either (i)~the surfaces $S$ and $S'$ are normal isotopic,
    or (ii)~$S$ and $S'$ can be made normal isotopic by adding or
    removing vertex linking components.
\end{corollary}

We proceed now to give a complete classification of joint vector
representations of embedded normal and octagonal almost normal surfaces.
As indicated earlier, one of the appealing features of joint
coordinates is that this classification corresponds precisely to Tollefson's
theorem for embedded \emph{normal} surfaces
(Theorem~\ref{t-admissible}), except for the fact that we must allow one
coordinate to become negative.

\begin{theorem} \label{t-joint-admissible}
    Let $\tri$ be a compact 3-manifold triangulation formed from
    $n$ tetrahedra.  An integer vector $\mathbf{w} \in \R^{3n}$
    is the joint vector representation of an embedded normal or
    octagonal almost normal surface in $\tri$ if and only if:
    \begin{itemize}
        \item At most one coordinate of $\mathbf{w}$ is negative;
        \item $\mathbf{w}$ satisfies the quadrilateral
        matching equations for $\tri$ (Definition~\ref{d-matchingquad});
        \item $\mathbf{w}$ satisfies the quadrilateral
        constraints for $\tri$ (Definition~\ref{d-quadconst});
        \item If there is a negative coordinate in $\mathbf{w}$,
        then this coordinate is set to $-1$.
    \end{itemize}
    Moreover, such a vector represents an embedded normal surface
    in $\tri$ if and only if all of its coordinates are non-negative.
\end{theorem}

It is worth pointing out that we interpret the quadrilateral matching
equations and the quadrilateral constraints literally for any
$3n$-dimensional vector.  We do not try to ``reconstruct''
quadrilateral coordinates from $\mathbf{w}$, but instead we read
Definitions~\ref{d-matchingquad} and~\ref{d-quadconst} precisely as given.
In particular, the vector
\[ \mathbf{w}~=\,\left(\ j_{1,1},j_{1,2},j_{1,3}\ ;
    \ \ldots,j_{n,3}\ \right) \in \R^{3n} \]
is deemed to satisfy the quadrilateral constraints if at most one
of $j_{i,1}$, $j_{i,2}$ and $j_{i,3}$ is non-zero for any given $i$.
Likewise, $\mathbf{w}$ satisfies the quadrilateral matching equations
if for each non-boundary edge $e$ of $\tri$ we have
\[ j_{i_1,u_1} + j_{i_2,u_2} + \ldots + j_{i_t,u_t} =
   j_{i_1,d_1} + j_{i_2,d_2} + \ldots + j_{i_t,d_t} ,\]
where each $u_k$ is the number of an upward quadrilateral type meeting $e$
in the $i_k$th tetrahedron of $\tri$, and each $d_k$ is the number of a
downward quadrilateral type meeting $e$ in the $i_k$th tetrahedron of $\tri$.

% We are some distance now from the original statement of the theorem.
% Redefine \prooflabel to make it clear what we're proving.
{\renewcommand{\prooflabel}{Proof of Theorem~\ref{t-joint-admissible}}
\begin{proof}
    Normal surfaces (as opposed to almost normal surfaces) are easily
    dealt with.  Suppose that $S$ is some embedded normal surface in $\tri$.
    Then we have $\jrep{S}=\qrep{S}$, and it is clear from
    Theorem~\ref{t-admissible} that $\jrep{S}$ satisfies the four
    conditions given in this theorem, and that every coordinate of
    $\jrep{S}$ is non-negative.
    Conversely, suppose that
    some integer vector $\mathbf{w} \in \R^{3n}$ satisfies these four
    conditions, and that all of its coordinates are non-negative.
    Then $\mathbf{w}$ satisfies the conditions of
    Theorem~\ref{t-admissible}, whereupon it follows that $\mathbf{w}$
    is the quadrilateral vector representation---and therefore also the joint
    vector representation---of some embedded normal surface in $\tri$.

    We turn our attention now to the more interesting case of
    octagonal almost normal surfaces.  The key observation is the
    following.  Consider the {\quadoct} matching equation derived from
    some non-boundary edge $e$ of the triangulation, as described in
    Definition~\ref{d-qo-matching}, and let $\Delta$ be some
    tetrahedron containing $e$.  If we use the numbering scheme of
    Definition~\ref{d-octquadnumber}, then the $i$th quadrilateral type
    in $\Delta$ is an \emph{upward} quadrilateral if and only if
    the $i$th octagon type in $\Delta$ is a \emph{downward} octagon,
    and vice versa.  This is easily verified by examining
    Figure~\ref{fig-matchingquadoct}.

    Using this observation, we can reduce each {\quadoct} matching
    equation~(\ref{eqn-qo-matching}) to the following:
    \begin{equation} \label{eqn-joint-matching-quadoct}
       (q_{i_1,u_1} - k_{i_1,u_1}) + \ldots + (q_{i_1,u_t} - k_{i_1,u_t}) =
       (q_{i_1,d_1} - k_{i_1,d_1}) + \ldots + (q_{i_1,d_t} - k_{i_1,d_t}),
    \end{equation}
    where the coordinates
    $q_{i_1,u_1},q_{i_2,u_2},\ldots,q_{i_t,u_t}$ and
    $q_{i_1,d_1},q_{i_2,d_2},\ldots,q_{i_t,d_t}$
    correspond to the upward and downward quadrilaterals respectively
    about the edge $e$.  Translated into joint coordinates, this reduces
    further to
    \begin{equation} \label{eqn-joint-matching}
        j_{i_1,u_1} + \ldots + j_{i_1,u_t} =
        j_{i_1,d_1} + \ldots + j_{i_1,d_t},
    \end{equation}
    which is identical to the corresponding quadrilateral matching
    equation in $\R^{3n}$.

    We can now finish the proof of Theorem~\ref{t-joint-admissible}.
    Suppose that $S$ is some octagonal almost normal surface in $\tri$.
    Then the following observations follow immediately
    from Theorem~\ref{t-qo-admissible}:
    \begin{itemize}
        \item Precisely one octagonal coordinate in $\krep{S}$ is
        non-zero, and the corresponding quadrilateral coordinate in
        $\krep{S}$ must be zero as a result.  Therefore precisely one
        coordinate of $\jrep{S}$ is negative.
        \item The {\quadoct} vector representation $\krep{S}$ satisfies
        each {\quadoct} matching equation as described
        by~(\ref{eqn-joint-matching-quadoct}).  Therefore
        the joint vector representation $\jrep{S}$ satisfies each
        \emph{quadrilateral} matching equation, as described
        by~(\ref{eqn-joint-matching}).
        \item For each tetrahedron of $\tri$, at most one of the six
        corresponding quadrilateral and octagonal coordinates in $\krep{S}$
        is non-zero, and so at most one of the three corresponding
        joint coordinates in $\jrep{S}$ is non-zero.  Therefore $\jrep{S}$
        satisfies the quadrilateral constraints.
        \item The unique non-zero octagonal coordinate in $\krep{S}$ has
        value $+1$, and so the unique negative coordinate in $\jrep{S}$
        has value $-1$.
    \end{itemize}
    Therefore the joint vector representation $\jrep{S}$ satisfies all
    four conditions listed in the statement of this theorem.

    Conversely, suppose that some integer vector
    \[ \mathbf{w}~=\,\left(\ j_{1,1},j_{1,2},j_{1,3}\ ;
        \ \ldots,j_{n,3}\ \right) \in \R^{3n} \]
    satisfies all four conditions listed in this theorem statement, and that
    one of its coordinates is negative (recalling that the non-negative case
    was dealt with at the beginning of this proof).  We define the
    $6n$-dimensional vector
    \[ \mathbf{w}'~=\,
        \left(\ q_{1,1},q_{1,2},q_{1,3},\ k_{1,1},k_{1,2},k_{1,3}\ ;
        \ \ldots,k_{n,3}\ \right) \in \R^{6n}\]
    by setting
    \[
        q_{i,t} = \left\{ \begin{array}{l@{\quad}l}
            j_{i,t} & \mbox{if $j_{i,t} \geq 0$;} \\
            0 & \mbox{if $j_{i,t} < 0$,}
        \end{array}\right.
        \quad\mbox{and}\quad
        k_{i,t} = \left\{ \begin{array}{l@{\quad}l}
            0 & \mbox{if $j_{i,t} \geq 0$;} \\
            -j_{i,t} & \mbox{if $j_{i,t} < 0$.}
        \end{array}\right.
    \]
    By using the four conditions of this theorem statement and following
    the previous argument in reverse, it is simple to show that
    $\mathbf{w}'$ satisfies the conditions of Theorem~\ref{t-qo-admissible}.
    It follows then that
    $\mathbf{w}'$ is the {\quadoct} vector representation of some
    octagonal almost normal surface in $\tri$, and so $\mathbf{w}$
    is the joint vector representation of this same surface.
\end{proof}
} % End redefined \prooflabel.

Because joint coordinates are $3n$-dimensional, we are able to visualise
them explicitly in $\R^3$ for a one-tetrahedron triangulation.  We do just
this in the following example to illustrate the various conditions of
Theorem~\ref{t-joint-admissible}.

\begin{example} \label{eg-sphere}
    Let $\trisphere$ be the following compact 3-manifold triangulation
    (which we will shortly prove represents the 3-sphere).
    We begin with the single tetrahedron $\mathit{ABCD}$, and identify faces
    $\mathit{ABC} \leftrightarrow \mathit{BCD}$ (with a twist) and
    $\mathit{ABD} \leftrightarrow \mathit{ACD}$ (folded directly over
    the common edge $\mathit{AD}$), as illustrated in
    Figure~\ref{fig-trisphere}.  The resulting
    triangulation has one tetrahedron, one vertex
    (since $A$, $B$, $C$ and $D$ are all identified), and two edges
    (where $\mathit{AB}$, $\mathit{BC}$, $\mathit{CD}$, $\mathit{BD}$
    and $\mathit{AC}$ are all identified, and $\mathit{AD}$ is left in
    a class of its own).

    \begin{figure}[htb]
    \centering
    \includegraphics{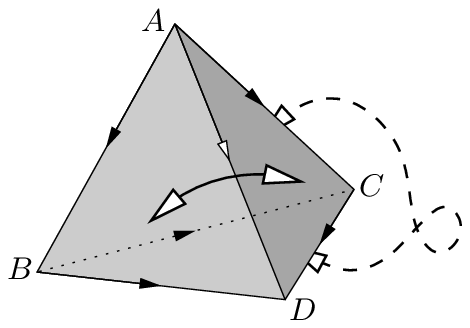}
    \caption{The one-tetrahedron triangulation $\trisphere$}
    \label{fig-trisphere}
    \end{figure}

    Let $\Delta$ represent the sole tetrahedron of $\trisphere$,
    and number the three quadrilateral types in $\Delta$ so that
    types 1, 2 and 3 separate the edge pairs
    ($\mathit{AB},\mathit{CD}$), ($\mathit{AC},\mathit{BD}$) and
    ($\mathit{AD},\mathit{BC}$) respectively.
    We find then that
    both joint matching equations for $\trisphere$ (one for each edge)
    reduce to the form $j_{1,1} = j_{1,2}$.

    We plot the resulting solution space in $\R^3$ in
    Figure~\ref{fig-solneg}.  In the top-left diagram, we shade the
    region in which $\mathbf{w}=(j_{1,1}, j_{1,2}, j_{1,3}) \in \R^3$ has
    at most one negative coordinate (for clarity, we restrict our shading
    to the interior of a sphere around the origin).
    In the top-right diagram, we shade the intersection of this region
    with the hyperplane $j_{1,1}=j_{1,2}$, which gives us the closed
    half-plane $H = \{(x,x,z)\,|\,x \geq 0\}$.  If we wish to enforce the
    quadrilateral constraints then we must restrict our attention to the
    three coordinate axes (where at most one coordinate is non-zero);
    the final intersection of $H$ with these three axes is plotted in the
    bottom-left diagram.  The resulting solution space is simply the
    entire $j_{1,3}$ axis, taken in both directions.

    \begin{figure}[htb]
    \centering
    \includegraphics[scale=0.9]{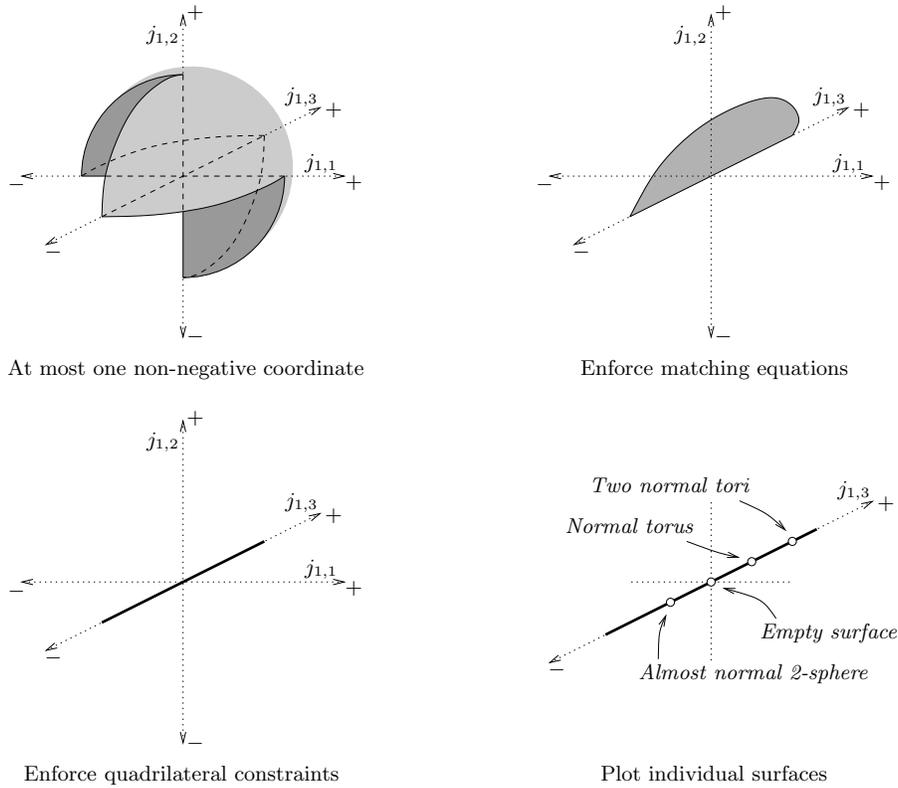}
    \caption{Building the solution space for the triangulation $\mathcal{S}$}
    \label{fig-solneg}
    \end{figure}

    From Theorem~\ref{t-joint-admissible} it follows that, if we ignore
    vertex linking components, then the normal and octagonal normal surfaces
    in $\trisphere$ correspond precisely to the integer points
    \[ \{(0,0,k)\,|\,k \in \Z,\ k \geq -1\}. \]
    With some further investigation we can classify these surfaces as
    follows:
    \begin{itemize}
        \item $(0,0,0)$ represents the empty surface;
        \item $(0,0,k)$ for $k \geq 1$ represents $k$ copies of the
        embedded normal torus surrounding the edge $\mathit{AD}$,
        which is referred to by Jaco and Rubinstein as a
        \emph{thin edge link} \cite{jaco03-0-efficiency};
        \item $(0,0,-1)$ is an octagonal almost normal 2-sphere.
    \end{itemize}
    These surfaces are individually plotted in the bottom-right diagram of
    Figure~\ref{fig-solneg}.

    To finish, we note that (i)~there are no normal 2-spheres (aside
    from the vertex link which we have ignored), and so $\trisphere$
    is a 0-efficient triangulation, and that (ii)~$\trisphere$ contains an
    octagonal almost normal 2-sphere.  Using Theorem~\ref{t-std-sphere}
    and noting that $\trisphere$ is orientable, it follows that $\trisphere$ is
    in fact a triangulation of the 3-sphere.
\end{example}

A natural question to ask at this point is what becomes of the
projective solution space in joint coordinates.  Recall that in other
coordinate systems, the non-negative orthant and the matching equations
intersect to give a convex polyhedral cone, and that the projective solution
space is a cross-section of this cone, taken by intersecting the cone with
the hyperplane $\sum x_i = 1$.

The difficulty we face with joint coordinates is that we no longer have
a convex polyhedral cone to work with.  Instead we begin with the union of
$3n+1$ orthants in $\R^{3n}$ (where at most one coordinate is
non-negative), which is not even a convex set.  Upon intersecting this
with the joint matching
equations, we obtain a set $\mathcal{P}$ with the following properties.
$\mathcal{P}$ is a cone in the sense that $\mathbf{x} \in \mathcal{P}$
implies that $\lambda \mathbf{x} \in \mathcal{P}$ for any $\lambda \geq 0$,
but like the union of orthants before it, $\mathcal{P}$ might not be
convex (although in Example~\ref{eg-sphere} it happens to be).
More importantly, $\mathcal{P}$ can contain diametrically opposite points
(such as $(0,0,\pm 1)$ in our example), and so in general we
cannot form a cross-section by slicing through $\mathcal{P}$ with a
hyperplane.

We could perhaps take a cross-section using the unit sphere, but this
would lift us out of the world of polytopes, making it difficult
to design algorithms.  Perhaps the simplest solution is
to take a cross section using the ``polyhedral unit sphere'' $\sum |x_i| = 1$,
as illustrated in Figure~\ref{fig-polysection}.
Continuing with Example~\ref{eg-sphere}, the left-hand diagram of
Figure~\ref{fig-polysection} shows the intersection
of our four original orthants with the ``sphere''
$|j_{1,1}| + |j_{1,2}| + |j_{1,3}| = 1$, and the right-hand
diagram shows the subsequent intersection with the matching equation
$j_{1,1} = j_{1,2}$.

\begin{figure}[htb]
\centering
\includegraphics[scale=0.9]{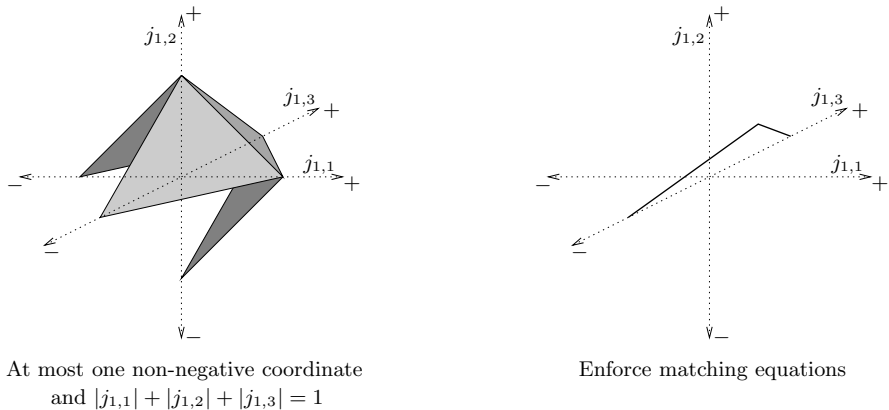}
\caption{Constructing the ``projective solution space'' in joint coordinates}
\label{fig-polysection}
\end{figure}

Although we are now able to define the projective solution space in a
sensible way, we do not obtain a single convex polytope as a result.
Instead we obtain
$3n+1$ distinct convex polytopes---one for each of the original
orthants---joined together along their boundaries.  To enumerate the vertices
of this structure would therefore require $3n+1$ distinct passes through
the vertex enumeration algorithm,\footnote{In fact only $3n$ passes are
required.  We can ignore the non-negative orthant, since it can be shown
that every vertex of the non-negative orthant is also a vertex of one
of the $3n$ adjacent ``almost non-negative'' orthants.}
bringing us back to the scheme of Casson that we discussed
in Section~\ref{s-enumeration-notes}.  It is worth noting again that
the polytope of this structure that sits within the non-negative orthant
is precisely Tollefson's quadrilateral projective solution space
for \emph{normal} surfaces.

It follows then that joint coordinates do not appear practical for use
in enumeration algorithms.  Nevertheless, they have appealing geometric
properties that may render them useful for other purposes:
\begin{itemize}
    \item They live in a remarkably small number of dimensions;
    \item They express the space of admissible vectors for octagonal
    almost normal surfaces as a natural geometric extension of
    Tollefson's space for normal surfaces, obtained simply by
    expanding our scope from the non-negative orthant to include
    the neighbouring ``almost non-negative'' orthants.
\end{itemize}

As an immediate application, these properties make joint coordinates
a useful tool for visualising the almost normal solution space.
More generally, they could perhaps open the way for new \emph{theoretical}
insights into the structure of the solution space.

To illustrate the latter point, we can draw analogies with Casson's edge
weight coordinates for normal surfaces, which are developed and
exploited in \cite{burton03-thesis}.  Like
joint coordinates, edge weight coordinates use very few dimensions and
are geometrically appealing, but a loss of convexity makes them
impractical for use in algorithms.
Nevertheless, their tight geometric structure has led to
new theoretical and combinatorial insights, and we hope that joint
coordinates can offer the same.

\bibliographystyle{amsplain}
\bibliography{pure}

\vspace{1cm}
\noindent
Benjamin A.~Burton \\
School of Mathematics and Physics, The University of Queensland \\
Brisbane QLD 4072, Australia \\
(bab@debian.org)

\end{document}